\documentclass{amsart}
\usepackage{amssymb,amsmath,amsfonts,amsthm,amsxtra,afterpage}
\usepackage[dvips,all]{xy}
\usepackage{xcolor}
\usepackage{mdframed}
\usepackage[normalem]{ulem}
\usepackage{enumerate}
\usepackage{enumerate}

\usepackage{todonotes}

\newtheorem{thm}{Theorem}[section]
\newtheorem{prop}[thm]{Proposition}
\newtheorem{lem}[thm]{Lemma}

\newtheorem{cor}[thm]{Corollary}

\theoremstyle{definition}
\newtheorem{dfn}[thm]{Definition}
\newtheorem{eg}[thm]{Example}

\theoremstyle{remark}
\newtheorem{rmk}[thm]{Remark}

\numberwithin{equation}{section}

\renewenvironment{proof}[1][\proofname]{\begin{trivlist}\item[\hskip \labelsep \itshape \bfseries #1{}\hspace{2ex}]}
{\qed\end{trivlist}}

\begin{document}
\title[Invariant Theory of $ \mathbb{G}_{a} $-Actions.]{On the Invariant Theory of $ \mathbb{G}_{a} $-Actions from a Geometric Perspective.}
\author{Stephen Maguire}
\email{maguire2@illinois.edu}

\begin{abstract}
    In this paper we give a strict classification of $ \mathbb{G}_{a} $-representations.  This is done through the notion of a $ c(t) $-pair.  Namely if $ \operatorname{Spec}(A) $ is a $ \mathbb{G}_{a} $-variety with action $ \beta $, then a $ c(t) $-pair is a pair of elements $ (g,h) $ such that $ g(t_{0} \ast x) = g(x)+c(t_{0}) h(x) $.  This allows us to describe exactly when an affine, $ \mathbb{G}_{a} $-stable, sub-variety $ D(h) $ is a trivial bundle over $ D(h)//\mathbb{G}_{a} $.  If $ \operatorname{Spec}(A) $ is a $ \mathbb{G}_{a} $-variety, we define the large pedestal ideal $ \mathfrak{P}_{g}(A) $ and the pedestal ideal $ \mathfrak{P}(A) $.  If $ \beta: \mathbb{G}_{a} \to \operatorname{GL}(\mathbf{V}) $ is a $ \mathbb{G}_{a} $-representation, then we classify such a representation on whether
    \begin{itemize}
        \item[a)] the large pedestal ideal $ \mathfrak{P}_{g}(S_{k}(\mathbf{V}^{\ast})) $ is equal to zero,
        \item[b)] the large pedestal ideal is non-zero, but the pedestal ideal is equal to zero, or
        \item[c)] the pedestal ideal is non-zero.
    \end{itemize}
    In case a) the ring of invariants is simply $ S_{k}\left((\mathbf{V}^{\ast})^{\mathbb{G}_{a}}\right) $ and these representations are un-interesting from the perspective of classical invariant theory.  Case c) is the nicest case in which, after a suitable modification, there is an open affine sub-variety $ U $ such that $ U $ is a trivial $ \mathbb{G}_{a} $-bundle over its image in $ \operatorname{Spec}(S_{k}(\mathbf{V}^{\ast})^{\mathbb{G}_{a}}) $.  If $ \operatorname{Spec}(A) $ is a $ \mathbb{G}_{a} $-variety, and $ \mathfrak{P}(A) $ is non-zero, then $ \operatorname{Spec}(A) $ is a quasi-principle $ \mathbb{G}_{a} $-variety.  Case b) is in some sense ``a mixture'' of cases a) and c).
    
    In the process, we generalize van den Essen's algorithm so that it works for quasi-principle $ \mathbb{G}_{a} $-actions over fields $ k $ of positive characteristic.  We also show how we can compute the ring $ S_{k}(\mathbf{V}^{\ast})^{\mathbb{G}_{a}}_{h(X)} $ for sufficient $ h(X) \in \mathfrak{P}_{g}(S_{k}(\mathbf{V}^{\ast})) $ in case b).
\end{abstract}

\maketitle
\section{Introduction}
There are many reasons to be interested in $ \mathbb{G}_{a} $-actions.  One such reason is the Classical Zariski conjecture, which asks ``if a $ \mathbb{A}^{1}_{k} \times Y \cong \mathbb{A}^{n}_{k} $, then is $ Y $ isomorphic to $ \mathbb{A}^{n-1}_{k} $?''  Since $ \mathbb{G}_{a} $ acts on $ \mathbb{A}^{1}_{k} \times Y $ via the natural action on the first component and a trivial action on $ Y $, one may answer the question in the affirmative if one can show that whenever $ \mathbb{G}_{a} $-acts on $ \operatorname{Spec}(k[x_{1},\dots,x_{n}]) $ so that it is a trivial $ \mathbb{G}_{a} $-bundle over its image in $ \operatorname{Spec}(k[x_{1},\dots,x_{n}]^{\mathbb{G}_{a}}) $, then $ \operatorname{Spec}(k[x_{1},\dots,x_{n}]^{\mathbb{G}_{a}}) \cong \mathbb{A}^{n-1}_{k} $.  This question is closely related to the Jacobian conjecture, and provides one reason to be interested in $ \mathbb{G}_{a} $-actions.

Another application of $ \mathbb{G}_{a} $-actions is in determining whether a variety is separably ruled or separably uniruled.  An $ n $-dimensional variety $ Z $ is separably uniruled if there is an $ n-1 $-dimensional variety $ Y $ such that there exists a dominant, separable, generically finite, rational map $ \phi: \mathbb{P}^{1}_{k} \times Y \dashrightarrow Z $ and it is separably ruled if the underlying rational map is a birational map.  If $ Z $ is proper over $ k $ and $ Y $ is complete, then there is a non-trivial open sub-variety $ U \subseteq Y $ and $ V \subseteq Z $ such that $ \phi: \mathbb{A}^{1}_{k} \times U \to V $ is a dominant, \'{E}tale morphism.  In many circumstances this is equivalent to the statement that there is an \'{E}tale cover $ W $ of $ V $ such that $ W $ is a trivial $ \mathbb{G}_{a} $-bundle.

A final application of $ \mathbb{G}_{a} $-actions is ``the Weitzenb\"{o}ck conjecture''.  Weitzenb\"{o}ck's theorem says that if $ \beta: \mathbb{G}_{a} \to \operatorname{GL}(\mathbf{V}) $ is a linear representation of $ \mathbb{G}_{a} $ over a field $ L $ of characteristic zero, and $ \{x_{1},\dots,x_{n}\} $ is a basis of $ \mathbf{V}^{\ast} $, then $ L[x_{1},\dots,x_{n}]^{\mathbb{G}_{a}} $ is a finitely generated $ L $-algebra.  While Roland Weitzenb\"{o}ck may not have conjectured whether Hilbert's 14th problem has an affirmative answer when the characteristic of the base field is $ p>0 $, we dub it ``the Weitzenb\"{o}ck conjecture'' for simplicity.

In characteristic zero, the existence of a $ \mathbb{G}_{a} $-action $ \beta: \mathbb{G}_{a} \times \operatorname{Spec}(A) \to \operatorname{Spec}(A) $ is equivalent to the existence of a locally nilpotent derivation $ \delta \in T_{A/k} $.  In this case, the co-action $ \beta^{\sharp} $ sends an element $ a $ to $ \exp(t \delta)(a)= \sum_{j=0}^{\infty} \left(\delta^{j}(a)/j!\right)t^{j} $ and the ring of invariants is equal to the kernel of $ \delta $.  In \cite[Chapter 1, First Principles, pg. 10]{FreudenbergDerivBook} Gene Freudenberg defines the plinth ideal of $ A^{\mathbb{G}_{a}} $ to be the ideal generated by elements of $ A^{\mathbb{G}_{a}} \cap \delta(A) $.  Freudenberg also defines a slice to be an element $ a \in A $ such that $ \delta(a) = 1 $.  The reason the term ``slice'' is used is that if such an element exists, then $ \operatorname{Spec}(A) \cong \mathbb{G}_{a} \times \operatorname{Spec}(A^{\mathbb{G}_{a}}) $.  Freudenberg defines an element $ a \in A $ to be a ``local slice'' if $ \delta(a) \ne 0 $ and $ a \in \ker(\delta^{2}) $.  In this instance, $ a/\delta(a) $ is a slice on $ \operatorname{Spec}(A_{\delta(a)}) $.  So the plinth ideal yields valuable information about local triviality of $ \mathbb{G}_{a} $-varieties.

In positive characteristic, a $ \mathbb{G}_{a} $-action $ \beta: \mathbb{G}_{a} \times \operatorname{Spec}(A) \to \operatorname{Spec}(A) $ exists if and only if there exists a locally finite, iterative, higher derivation, $ \{\phi_{j}\}_{j \in \mathbb{N}_{0}} $.  In positive characteristic a locally finite, iterative, higher derivation $ \{\phi_{j}\}_{j \in \mathbb{N}_{0}} $ is not uniquely determined by a derivation $ \phi_{1} $.  Also, up to scalar multiplication, there is only one additive polynomial in $ L[t] $ if $ L $ is a field of characteristic zero.  However, if $ k $ is a field of positive characteristic, then the additive polynomials of the ring $ k[t] $ form a non-commutative, Euclidean ideal domain known as the Ore ring.  As a result, it becomes remarkably more difficult to formulate a local triviality criterion.  If $ \{\phi_{j}\}_{j \in \mathbb{N}_{0}} $ is a locally finite, iterative, higher derivation on $ A $, then there is a corresponding action $ \beta $ whose co-action $ \beta^{\sharp} $ sends $ a \in A $ to $ \sum_{j=0}^{\infty} \phi_{j}(a)t^{j} $.  In \cite{Kuroda} Shigeru Kuroda defines $ \deg_{\phi}(a) $ and $ \operatorname{lc}_{\phi}(a) $ to be the degree and leading coefficient of $ \beta^{\sharp}(a) $ as a polynomial in $ \operatorname{Frac}(A)[t] $.  Kuroda then defines a local slice to be an element $ a \in A \setminus A^{\mathbb{G}_{a}} $ such that $ \deg_{\phi}(a) $ is minimal for all such elements.

Let $ b(t),c(t) $ be two $ k $-linearly independent, additive polynomials such that $ \deg(b(t))<\deg(c(t)) $.  If $ \beta: \mathbb{G}_{a} \to \operatorname{GL}(\mathbf{V}) $ and $ \{x_{1},x_{2},x_{3}\} $ is a dual basis of the three dimensional, vector space $ \mathbf{V} $, such that $ \beta^{\sharp} $ is described below:
\begin{align}
    x_{1} & \mapsto x_{1}, \notag \\
    x_{2} & \mapsto x_{2}, \notag \\
    x_{3} & \mapsto x_{3}+b(t)x_{1}+c(t)x_{2} \label{E:1}
\end{align}
then even though $ x_{3} $ is a local slice under Kuroda's terminology, the sub-variety $ D(x_{2}) $ is not isomorphic to $ \mathbb{G}_{a} \times \operatorname{Spec}(k[x_{1},x_{2}]_{x_{2}}) $.  In particular, the variety $ D(x_{2}) $ cannot be endowed with the structure of a trivial $ \mathbb{G}_{a} $-bundle even though $ D(x_{2}) $ is isomorphic to $ \operatorname{Spec}(k[x_{3}/x_{2}]) \times \operatorname{Spec}(k[x_{1},x_{2}]_{x_{2}}) $.  The underlying problem is that the isomorphism is not $ \mathbb{G}_{a} $-equivariant.  As a result, the term ``local slice'' is a misnomer.

For this reason we create the notion of a $ c(t) $-pair.  If $ \beta: \mathbb{G}_{a} \times \operatorname{Spec}(A) \to \operatorname{Spec}(A) $ is an action of $ \mathbb{G}_{a} $ over an algebraically closed field $ k $ of positive characteristic $ p>0 $, and $ c(t) $ is an additive polynomial of $ k[t] $, then a $ c(t) $-pair is a pair of elements $ (g,h) $ such that $ g(t_{0} \ast x) = g(x)+c(t_{0})h(x) $ for any closed point $ (t_{0},x) \in \mathbb{G}_{a} \times \operatorname{Spec}(A) $.  A $ b(t) $-pair is called a quasi-principle pair if $ \mathbf{ker}(b(t)) $ stabilizes all of $ \operatorname{Spec}(A) $ and is a principle pair if $ b(t) $ is equal to $ t $.  The existence of a principle pair is integrally related to local triviality.  If a $ b(t) $-pair is a quasi-principle pair, then one may often assume that $ b(t) $ is equal to $ t $ by replacing $ \mathbb{G}_{a} $ by $ \mathbb{G}_{a}//\mathbf{ker}(b(t)) \cong \mathbb{G}_{a} $.  We define the large pedestal ideal $ \mathfrak{P}_{g}(A) $ to be the ideal of $ A $ generated by $ \{ h $ such that there exists a non-zero, additive polynomial $ c(t) $ and a $ g \in A $ such that $ (g,h) $ is a $ c(t) $-pair $ \} $.  We define the pedestal ideal $ \mathfrak{P}(A) $ to be the ideal generated by $ \{ 0 $ and all $ h $ such that there exists a non-zero, additive polynomial $ b(t) $ and a $ g \in A $ such that $ (g,h) $ is a quasi-principle $ b(t) $-pair $ \} $.

We describe the theory of the pedestal ideal in intricate detail in this paper and explicitly categorize the types of pathologies which occur in representations like the one described in ~\eqref{E:1}.  Namely we show that the only time such pathologies may occur for an action $ \beta: \mathbb{G}_{a} \times \operatorname{Spec}(A) \to \operatorname{Spec}(A) $ is when $ \mathfrak{P}_{g}(A) $ is equal to zero.  Let $ k $ be an algebraically closed field of characteristic greater than two and let $ c_{1}(t),c_{2}(t),c_{3}(t) $ be $ k $-linearly independent, additive polynomials such that $ c_{1}(t) $ is not equal to $ t $.  If $ \{x_{1},\dots,x_{5}\} $ is a dual basis of $ \mathbf{V} $ and $ \beta: \mathbb{G}_{a} \to \operatorname{GL}(\mathbf{V}) $ is the five dimensional, representation with co-action below:
\begin{align}
    x_{1} & \mapsto x_{1}, \notag \\
    x_{2} & \mapsto x_{2}, \notag \\
    x_{3} & \mapsto x_{3}+c_{1}(t)x_{1}, \notag \\
    x_{4} & \mapsto x_{4}+c_{1}(t)x_{2}, \notag \\
    x_{5} & \mapsto x_{5}+c_{1}(t)(x_{3}+x_{4})+(c_{1}(t)^{2}/2+c_{2}(t))(x_{1})+(c_{1}(t)^{2}/2+c_{3}(t))x_{2}, \label{E:89}
\end{align}
then there is no open sub-variety $ U \subseteq \mathbf{V} $ such that $ U $ is a trivial bundle over $ U//\mathbb{G}_{a} $, but $ \mathfrak{P}_{g}(k[X]) $ is non-zero.  In ~\eqref{E:89} and representations with similar pathology, the large pedestal ideal is non-zero while the pedestal ideal is equal to zero.  We describe in Theorem ~\ref{T:trivialPedForm} (see page \pageref{T:trivialPedForm}) when this happens for representations, and we show how to locally compute the ring of invariants.  For example, it is possible to compute $ k[X]^{\mathbb{G}_{a}}_{x_{1}} $ and $ k[X]^{\mathbb{G}_{a}}_{x_{2}} $ provided adequate knowledge of the invariant theory of finite groups whose order is a power of $ p $.

If a $ \mathbb{G}_{a} $ action $ \beta: \mathbb{G}_{a} \times \operatorname{Spec}(A) \to \operatorname{Spec}(A) $ is locally principle, then there is always a principle pair $ (g,h) $.  In \cite[Chapter 1, Section 2, pg. 15]{FreudenbergDerivBook} Freudenberg defines a homomorphism from $ A_{h} \to A^{\mathbb{G}_{a}}_{h} $ as follows.  If $ \operatorname{Spec}(A) $ is a $ \mathbb{G}_{a} $-variety over a field $ L $ of characteristic zero with action $ \beta $ and $ (g,h) $ is a principle pair, then this map sends $ a $ to $ \beta^{\sharp}(a) \mid_{t=-g/h} $.  We show that if $ \operatorname{Spec}(A) $ is a $ \mathbb{G}_{a} $-variety over a field $ k $ of positive characteristic, and $ (g,h) $ is a principle pair, then the map which sends $ a $ to $ \beta^{\sharp}(a) \mid_{t=-g/h} $ is also a ring homomorphism from $ A_{h} $ to $ A_{h}^{\mathbb{G}_{a}} $.
\section{Basic Definitions}
\begin{dfn} \label{D:ore}
    Let $ k $ be a field of positive characteristic $ p > 0 $.  The \emph{Ore ring} is the ring whose underlying set is the set of additive polynomials of $ k[t] $.  Addition in the Ore ring is point-wise addition while multiplication is composition.
\end{dfn}\
There is another way to view the Ore ring.  If $ F^{\ell} $ corresponds to the $ \ell $-th iterate of the Frobenius morphism, then $ t^{p^{\ell}} = F^{\ell}(t) $.  The underlying set of the Ore ring is $ k[F] $, and addition is the same as in the polynomial ring $ k[F] $.  However, $ aF^{j} \cdot (b F^{i}) = a b^{p^{j}} F^{i+j} $.  If $ b(F) \in k[F] $, then the map which sends $ b(F) $ to $ b(F) \circ t $ is an isomorphism between our two conceptualizations of the Ore ring.  We shall denote the Ore ring by $ \mathfrak{O} $.

By a proof analogous to the proof used to show that $ k[t] $ is a Euclidean Ideal domain, it is possible to show that there is a right division algorithm for the Ore ring.  If $ k $ is a perfect field, then there is also a left division algorithm (see \cite[Proposition 1.6.2, Proposition 1.6.5]{Goss} and \cite[Chapter I, Theorem 1, pg. 562]{Ore}).  As a result, the Ore ring is a non-commutative Euclidean ideal domain.

If $ k $ is algebraically closed, then $ k $ is perfect.  So for $ k $ algebraically closed, the Ore ring is a non-commutative Euclidean Ring, and so for any two additive polynomials $ b_{1}(t) $ and $ b_{2}(t) $ a greatest common divisor exists and is the same regardless of whether the ideal generated by $ b_{1}(t) $ and $ b_{2}(t) $ is a left or right ideal.  
    
The greatest common divisor of two additive polynomials $ b_{1}(t) $ and $ b_{2}(t) $ in $ \mathfrak{O} $ is not the same as their greatest common divisor in $ k[t] $.  We will denote the greatest common divisor of $ b_{1}(t) $ and $ b_{2}(t) $ in $ \mathfrak{O} $ by $ \mathfrak{O}(b_{1}(t),b_{2}(t)) $.  For two additive polynomials $ c_{1}(t) $ and $ c_{2}(t) $, if $ b(t) $ is equal to $ \mathfrak{O}(c_{1}(t),c_{2}(t)) $, then there are additive polynomials $ d_{1}(t) $ and $ d_{2}(t) $ such that:
\begin{align*}
    d_{1}(b(t)) &= c_{1}(t) \\
    d_{2}(b(t)) &= c_{2}(t).
\end{align*}
This is not necessarily the case for the greatest common divisor of $ c_{1}(t) $ and $ c_{2}(t) $ in $ k[t] $.
\begin{eg} \label{EG:boft}
    Let $ k $ be an algebraically closed field, and $ b(t) $ an additive polynomial.  As a variety, $ \mathbf{ker}(b(t)) $ is equal to $ \operatorname{Spec}(k[t]/\langle b(t) \rangle) $.  It inherits the Hopf algebra structure from $ \mathbb{G}_{a} $.  One should note that if $ b(t) $ contains no linear term then $ \mathbf{ker}(b(t)) $ is a non-reduced, affine group scheme.  Note that one also obtains an endomorphism of $ \mathbb{G}_{a} $ from the ring homomorphism which sends $ t $ to $ b(t) $.  From this we obtain the following exact sequence:
    \begin{equation*}
    \xymatrix{
        0 \ar[r] & \mathbf{ker}(b(t)) \ar[r] & \mathbb{G}_{a} \ar[r]^{b(t)} & \mathbb{G}_{a} \ar[r] & 0
        }.
    \end{equation*}
\end{eg}
\begin{lem} \label{L:gaPolys}
    Let $ \beta: \mathbb{G}_{a} \to \operatorname{GL}(\mathbf{V}) $ be a linear representation of dimension $ n $ over an algebraically closed field $ k $, and let $ \{x_{1},\dots,x_{n}\} $ be an upper triangular basis of $ \mathbf{V}^{\ast} $, i.e, a basis such that $ \beta^{\sharp}(x_{i}) \in k[x_{1},\dots,x_{i}][t] $.  Let $ q_{i,j}(t) \in k[t] $ be polynomials such that $ \beta^{\sharp}(x_{i}) = x_{i}+\sum_{j=1}^{i-1} q_{i,j}(t)x_{j} $ for $ i>1 $.  If we denote $ t \otimes 1 $ by $ t_{1} $ and $ 1 \otimes t $ by $ t_{2} $, then $ q_{i,j}(t_{1}+t_{2})-q_{i,j}(t_{1})-q_{i,j}(t_{2}) =\sum_{s=j+1}^{i-1} q_{i,s}(t_{1})q_{s,j}(t_{2}) $ for $ j<i-1 $ and $ q_{i,i-1}(t) $ is an additive polynomial.
\end{lem}
\begin{proof}
    Because $ \beta $ is an action
    \begin{multline} \label{E:42}
        x_{i}+\sum_{j=1}^{i-1} q_{i,j}(t_{1}+t_{2})x_{j} = (\operatorname{id}_{k[X]} \otimes \mu_{\mathbb{G}_{a}}^{\sharp})\circ \beta^{\sharp}(x_{i}) \\
        = (\beta^{\sharp} \otimes \operatorname{id}_{k[t]}) \circ \beta^{\sharp}(x_{i}) \\
        = (\beta^{\sharp} \otimes \operatorname{id}_{k[t]})\left(x_{i}+\sum_{s=1}^{i-1} q_{i,s}(t)x_{s}\right) \\
        = x_{i}+\sum_{s=1}^{i-1} q_{i,s}(t_{1})\left(x_{s}+\sum_{u=1}^{s-1} q_{s,u}(t_{2})x_{u}\right) + \sum_{\ell=1}^{i-1} q_{i,\ell}(t_{2})x_{\ell}.
    \end{multline}
    The lemma now follows by equating the coefficient of $ x_{j} $ in both sides of ~\eqref{E:42}.
\end{proof}
\begin{dfn} \label{D:aPair}
    Let $ b(t) $ be an additive polynomial of $ k[t] $, and let $ \operatorname{Spec}(A) $ be a $ \mathbb{G}_{a} $-variety with action $ \beta $.  If there are elements $ g \in A $ and $ h \in A^{\mathbb{G}_{a}} $, such that $ \beta^{\sharp}(g) = g+b(t) h $, then $ (g, h) $ is \emph{a $ b(t) $-pair}.  Another way to say this is that if $ x \in \operatorname{Spec}(A) $ and $ t_{0} \in \mathbb{G}_{a} $, then
    \begin{equation*}
        g(t_{0} \ast x) = g(x)+b(t_{0})h(x).
    \end{equation*}
    If $ b(t) $ is equal to $ t $, then $ (g,h) $ \emph{is a principle pair}.
\end{dfn}
\begin{dfn} \label{D:principleAPair}
    If $ \operatorname{Spec}(A) $ is a $ \mathbb{G}_{a} $-variety over a field $ k $ with action $ \beta $, then a $ b(t) $-pair $ (g,h) $ is \emph{a quasi-principle $ b(t) $-pair} if $ \mathbf{ker}(b(t)) $ acts trivially on all of $ \operatorname{Spec}(A) $.
\end{dfn}
\begin{dfn} \label{D:pseudoPair}
    Let $ H $ be a sub group scheme of $ \mathbb{G}_{a} $, and let $ \operatorname{Spec}(A) $ be a $ \mathbb{G}_{a} $-variety.  If $ g $ is an element of $ A $, then let $ B $ be the ring generated by all $ H $-translates of $ g $.  The variety $ \operatorname{Spec}(B) $ is an $ H $-variety.  There is a representation $ \gamma: H \to \operatorname{GL}(\mathbf{V}) $ and an $ H $-equivariant, closed immersion $ \tau: \operatorname{Spec}(B) \to \mathbf{V} $.  If $ \mathbf{V} $ has the additional property that the image of $ \tau $ is not contained in any affine hyperplane of $ \mathbf{V} $, then the \emph{variance of $ g $ by $ H $} is the dimension of $ \mathbf{V} $.  We will use \emph{$ \operatorname{var}^{H}(g) $} to denote the variance of $ g $ by $ H $.  If $ H $ is $ \mathbb{G}_{a} $, then we will simply speak of the variance of $ g $ and denote it by $ \operatorname{var}(g) $.
\end{dfn}
\begin{cor} \label{Cor:pairSubSpace}
    Let $ \operatorname{Spec}(A) $ be a $ \mathbb{G}_{a} $-variety with action $ \beta $.  If $ g \in A $, then there is a $ b(t) \in \mathfrak{O} $ and an $ h \in A^{\mathbb{G}_{a}} $ such that $ (g,h) $ is a $ b(t) $-pair if and only if the variance of $ g $ is two.
\end{cor}
\section{The Geometry of Pairs.}
\subsection{Results on Pairs.}
The first lemma will give some equivalent conditions for what it means for a $ b(t) $-pair $ (g,h) $ to be quasi-principle.
\begin{lem} \label{L:principlePairEquiv}
    Let $ \operatorname{Spec}(A) $ be a $ \mathbb{G}_{a} $-variety with action $ \beta $, and recall the definition of the group scheme $ \mathbf{ker}(b(t)) $ found in Example ~\ref{EG:boft} (see page \pageref{EG:boft}).  The following statements are equivalent:
    \begin{itemize}
        \item[a)] the $ b(t) $-pair $ (g,h) $ is a non-trivial, quasi-principle, $ b(t) $-pair,
        \item[b)] there is a non-trivial, $ b(t) $-pair $ (g,h) $, and $ \beta $ factors through $ (b(t),\operatorname{id}_{\operatorname{Spec}(A)}) $, where $ b(t) $ is the endomorphism of $ \mathbb{G}_{a} $ obtained from the ring homomorphism which sends $ t $ to $ b(t) $,
        \item[c)] the following diagram commutes:
        \begin{equation} \label{E:47}
            \xymatrix{
                \mathbf{ker}(b(t)) \times \operatorname{Spec}(A) \ar[rr]^{p_{2}} \ar[d] & & \operatorname{Spec}(A) \ar[d] \\
                \mathbb{G}_{a} \times \operatorname{Spec}(A) \ar[rr]^{\beta} & & \operatorname{Spec}(A)
            },
        \end{equation} and $ (g,h) $ is a non-trivial, $ b(t) $-pair, i.e.,
        \begin{equation*}
            \beta^{\sharp}(g)=g+b(t)h.
        \end{equation*}
        \item[d)] there is an action $ \widetilde{\beta} $ such that the following diagram commutes:
        \begin{equation*}
        \xymatrix{
            \mathbf{ker}(b(t)) \times \operatorname{Spec}(A) \ar[r]^{p_{2}} \ar[d] & \operatorname{Spec}(A) \ar@{-}[d] \\
            \mathbb{G}_{a} \times \operatorname{Spec}(A) \ar[r]^{\beta} \ar[d]^{(b(t), \operatorname{id}_{\operatorname{Spec}(A)})} & \operatorname{Spec}(A) \ar@{-}[d] \\
            \mathbb{G}_{a} \times \operatorname{Spec}(A) \ar[r]^{\widetilde{\beta}} \ar[r] & \operatorname{Spec}(A)
        },
        \end{equation*}
        and $ (g,h) $ is a non-trivial, $ b(t) $-pair, i.e.
        \begin{equation*}
            \beta^{\sharp}(g)=g+b(t)h.
        \end{equation*}
        Moreover, if any of conditions a)-d) hold, then $ \mathbf{ker}(b(t)) $ is the largest, additive, sub-group scheme of $ \mathbb{G}_{a} $ which acts trivially on $ \operatorname{Spec}(A) $ under the action $ \beta $.
    \end{itemize}
\end{lem}
\begin{proof}
    If $ (g,h) $ is a non-trivial, quasi-principle $ b(t) $-pair, then $ \mathbf{ker}(b(t)) $ acts trivially on $ \operatorname{Spec}(A) $.  Let $ t_{0} \in \mathbf{ker}(b(t)) $, let $ y $ be a closed point of $ \operatorname{Spec}(A) $, and let $ f $ be an element of $ A $.  Since $ \mathbf{ker}(b(t)) $ acts trivially upon $ \operatorname{Spec}(A) $,
    \begin{align*}
        \beta^{\sharp}(f)(t_{0},y) &= f(t_{0} \ast y) \\
        &= f(y).
    \end{align*}
    So, $ \beta^{\sharp}(f)-f \in \langle b(t) \rangle A[t] $.
    
    We claim that if $ \beta^{\sharp}(f)-f \in \langle b(t) \rangle A[t] $ for all functions $ f \in A $, then $ \beta^{\sharp}(A) $ is contained in $ A[b(t)] $, i.e., $ \beta $ factors through $ (b(t),\operatorname{id}_{\operatorname{Spec}(A)}) $.  
    
    We shall first prove this for linear representations $ \beta: \mathbb{G}_{a} \to \operatorname{GL}(\mathbf{V}) $.  In doing so, we shall induce on the length of the socle series of $ \mathbf{V} $.  Assume that the length of the socle series of $ \mathbf{V} $ is two, $ \{x_{1},\dots,x_{n}\} $ is an upper triangular, dual basis of $ \mathbf{V} $ and $ \beta^{\sharp}(f(X))-f(X) \in \langle b(t) \rangle k[X][t] $.  If the length of the socle series is two, and $ \{x_{1},\dots,x_{s}\} $ is a basis of $ (\mathbf{V}^{\ast})^{\mathbb{G}_{a}} $, then for every $ i>s $, there are additive polynomials $ c_{i,j}(t) $ for $ 1 \le j \le s $ such that $ \beta^{\sharp}(x_{i}) $ is equal to $ x_{i}+\sum_{j=1}^{s} c_{i,j}(t)x_{j} $.  As a result, if $ \beta^{\sharp}(x_{i})-x_{i} \in \langle b(t) \rangle k[X][t] $, then $ c_{i,j}(t) \in \mathfrak{O}\langle b(t) \rangle $, since the zeroes of an additive polynomial form a sub-group under addition.
    
    Assume that if $ \beta: \mathbb{G}_{a} \to \operatorname{GL}(\mathbf{V}) $ is a linear representation, $ \{x_{1},\dots,x_{n}\} $ is an upper triangular, dual basis of $ \mathbf{V} $, the length of the socle series of $ \mathbf{V} $ is $ \ell<L $, and $ \beta^{\sharp}(f(X))-f(X) \in \langle b(t) \rangle k[X][t] $, then $ \beta^{\sharp}(k[X]) \subseteq k[X][b(t)] $.
    
    Now assume that $ \beta: \mathbb{G}_{a} \to \operatorname{GL}(\mathbf{V}) $ is a linear representation, $ \{x_{1},\dots,x_{n}\} $ is an upper triangular, dual basis of $ \mathbf{V} $, the length of the socle series of $ \mathbf{V} $ is $ L $, and $ \beta^{\sharp}(f(X))-f(X) \in \langle b(t) \rangle k[X][t] $ for all polynomials $ f(X) $.  If we apply the induction hypothesis to the dual representations of $ \mathbf{V}^{\ast}/(\mathbf{V}^{\ast})^{\mathbb{G}_{a}} $ and $ \operatorname{soc}^{2}(\mathbf{V}^{\ast}) $ respectively, then we prove that $ \beta^{\sharp}(k[X]) $ is contained in $ k[X][b(t)] $.  By induction if $ \beta: \mathbb{G}_{a} \to \operatorname{GL}(\mathbf{V}) $ is a linear representation with dual basis $ \{x_{1},\dots,x_{n}\} $ of $ \mathbf{V} $ such that $ \beta^{\sharp}(x_{i})-x_{i} \in \langle b(t) \rangle k[X][t] $, then $ \beta^{\sharp}(k[X]) $ is contained in $ k[X][b(t)] $.
    
    If $ \operatorname{Spec}(A) $ is a $ \mathbb{G}_{a} $-variety with action $ \beta $ such that $ \beta^{\sharp}(f)-f \in \langle b(t) \rangle A[t] $ for all $ f \in A $, then there is a linear representation $ \gamma: \mathbb{G}_{a} \to \operatorname{GL}(\mathbf{V}) $ and a $ \mathbb{G}_{a} $-equivariant immersion $ \tau: \operatorname{Spec}(A) \to \mathbf{V} $ such that $ \operatorname{Spec}(A) $ is not contained in any affine hyperplane.  
    
    If $ \{x_{1},\dots,x_{n}\} $ is a dual basis of $ \mathbf{V} $, then $ k[\tau^{\sharp}(x_{1}),\dots,\tau^{\sharp}(x_{n})] $ is isomorphic to $ A $ as a $ k $-algebra.  Since $ \operatorname{Spec}(A) $ is not contained in any affine hyperplane, there is no element $ x \in \mathbf{V}^{\ast} $ such that $ \tau^{\sharp}(x) $ is equal to zero.  As a result, \linebreak $ \gamma^{\sharp}(x_{i})-x_{i} \in \langle b(t) \rangle k[X][t] $, since $ \beta^{\sharp}(\tau^{\sharp}(x_{i}))-\tau^{\sharp}(x_{i}) \in \langle b(t) \rangle A[t] $.  Therefore, $ \gamma^{\sharp}(k[X]) \subseteq k[X][b(t)] $, which in turn means that
    \begin{align*}
        \beta^{\sharp}(A) &\subseteq (\tau^{\sharp} \otimes \operatorname{id}_{k[t]}) \circ \gamma^{\sharp}(k[X]) \\
        &\subseteq (\tau^{\sharp} \otimes \operatorname{id}_{k[t]})(k[X][b(t)]) \\
        &= A[b(t)].
    \end{align*}
    Since $ \beta^{\sharp}(A) \subseteq A[b(t)] $, the morphism $ \beta $ factors through $ (\operatorname{id}_{\operatorname{Spec}(A)}, b(t)) $.

    Assume that b) holds.  The statement that $ \beta $ factors through $ (\operatorname{id}_{\operatorname{Spec}(A)},b(t)) $ is equivalent to the statement that $ \beta^{\sharp}(A) \subseteq A[b(t)] $.  Let $ A $ equal $ k[z_{1},\dots,z_{n}] $ as a $ k $-algebra.  If $ t_{0} \in \mathbf{ker}(b(t)) $ and $ x \in \operatorname{Spec}(A) $, then
    \begin{align*}
        t_{0} \ast x &= (z_{1}(t_{0} \ast x),\dots,z_{n}(t_{0} \ast x)) \\
        &= (\beta^{\sharp}(z_{1})(t_{0},x),\dots,\beta^{\sharp}(z_{n})(t_{0},x)) \\
        &= (z_{1}(x),\dots,z_{n}(x)) \\
        &= x.
    \end{align*}
    This is because $ \beta^{\sharp}(z_{i})-z_{i} \in \langle b(t) \rangle A[b(t)] $.  Therefore $ \mathbf{ker}(b(t)) $ stabilizes all points of $ \operatorname{Spec}(A) $.  So a $ b(t) $-pair $ (g,h) $ is quasi-principle.  As a result b) holds if and only if a) holds. 

    If c) holds, then the diagram in ~\eqref{E:47} shows that $ \mathbf{ker}(b(t)) $ acts trivially on $ \operatorname{Spec}(A) $.  Since $ (g,h) $ is a $ b(t) $-pair, it is a quasi-principle $ b(t) $-pair.  As a result, c) implies a).

    If $ (g,h) $ is a quasi-principle $ b(t) $-pair, then $ \mathbf{ker}(b(t)) $ acts trivially on $ \operatorname{Spec}(A) $.  Therefore, the diagram in ~\eqref{E:47} commutes and c) holds.  Therefore c) holds if and only if $ (g,h) $ is a quasi-principle $ b(t) $-pair.

    It is clear that c) holds if and only if d) holds.

    If $ c(t) $ is an additive polynomial such that $ \mathbf{ker}(b(t)) \subseteq \mathbf{ker}(c(t)) $ and $ \mathbf{ker}(c(t)) $ acts trivially on $ \operatorname{Spec}(A) $, then $ c(t) $ is equal to $ c_{1}(b(t)) $ for some additive polynomial $ c_{1}(t) $.  Let $ t_{0} $ be a non-zero point of $ \mathbb{G}_{a} $ such that $ c_{1}(t_{0}) $ is equal to zero, and let $ t_{1} $ be a point in $ \mathcal{V}(\langle b(t)-t_{0} \rangle) $.  If $ c_{1}(t) $ is not equal to $ t $, then there exist such points $ t_{0} $ and $ t_{1} $.  The point $ t_{1} \in \mathbf{ker}(c(t)) $ because
    \begin{align*}
        c(t_{1}) &= c_{1}(b(t_{1})) \\
        &= c_{1}(t_{0}) \\
        &=0.
    \end{align*}
    However, we claim that $ t_{1} $ does not act trivially on $ \operatorname{Spec}(A) $.  If $ x \in D(h) $, then
    \begin{align*}
        g(t_{1} \ast x) &= g(x)+b(t_{1})h(x) \\
        &= g(x)+t_{0}h(x) \\
        &\ne g(x).
    \end{align*}
    This contradicts our assumption that $ \mathbf{ker}(c(t)) $ acts trivially on $ \operatorname{Spec}(A) $.  As a result, $ c(t) $ is equal to $ b(t) $ and $ \mathbf{ker}(b(t)) $ is the largest, additive, sub-group scheme which acts trivially on $ \operatorname{Spec}(A) $.
\end{proof}
\begin{dfn} ~\label{D:twistEquivt}
    The variety $ \mathbb{A}^{1}_{k} \cong \operatorname{Spec}(k[s]) $ has several $ \mathbb{G}_{a} $-actions.  If $ x \in \mathbb{A}^{1}_{k} $ and $ t_{0} \in \mathbb{G}_{a} $, then $ \gamma_{t} $ sends $ (t_{0},x) $ to $ x+t_{0} $.  However, if $ x \in \mathbb{A}^{1}_{k} $, $ t_{0} \in \mathbb{G}_{a} $ and $ c(t) $ is an additive polynomial, then $ \gamma_{c(t)} $ sends $ (t_{0},x) $ to $ x+c(t_{0}) $.  This is the $ c(t) $-twisted action on $ \mathbb{A}^{1}_{k} $.  If $ \mathbb{A}^{1}_{k} $ is endowed with this action, then we shall denote it by $ \mathbb{G}_{a}^{c(t)} $.

    If $ c(t) \in \mathfrak{O} $, then there is a $ \mathbb{G}_{a} $-equivariant morphism
    \begin{equation*}
        c(t): \mathbb{G}_{a}^{d(t)} \to \mathbb{G}_{a}^{c(d(t))},
    \end{equation*}
    which sends a point $ w $ to $ c(w) $.

    Likewise, if $ X $ is a $ \mathbb{G}_{a} $-variety, then $ X^{c(t)} $ is the space $ X $ where the action of $ \mathbb{G}_{a} $ has been twisted by the morphism $ c $.  If $ \beta $ is the action of $ \mathbb{G}_{a} $ on $ X $, then $ \beta_{c(t)} $ is the action of $ \mathbb{G}_{a} $ on $ X^{c(t)} $.
\end{dfn}
\begin{prop} \label{P:specificTrivialBundle}
    Let $ \operatorname{Spec}(A) $ be a $ \mathbb{G}_{a} $-variety with action $ \beta $ and let $ h \in A^{\mathbb{G}_{a}} $ be a non-zero invariant.  The following conditions are equivalent
    \begin{itemize}
        \item[a)] there is a dominant $ \mathbb{G}_{a} $-equivariant morphism $ \Phi: D(h) \to \mathbb{G}_{a}^{c(t)} $ (see Definition ~\ref{D:twistEquivt}).
        \item[b)] there is a $ c(t) $-pair $ (g,h^{e}) $ for some non-zero $ g \in A $, and $ e \in \mathbb{N}_{0} $.
    \end{itemize}
\end{prop}
\begin{proof}
    Assume that b) holds, i.e., $ (g,h^{e}) $ is a $ c(t) $-pair for a non-zero $ g \in A $.  Let $ \Phi: D(h) \to \mathbb{G}_{a}^{c(t)} $ be the morphism of varieties which sends a point $ x $ to $ (g/h^{e})(x) $.  We do not yet know if $ \Phi $ is $ \mathbb{G}_{a} $-equivariant.  If $ x \in D(h) $ and $ t_{0} \in \mathbb{G}_{a} $, then because $ (g,h^{e}) $ is a $ c(t) $-pair:
    \begin{align*}
            \Phi \circ \beta(t_{0},x) &= \Phi(t_{0} \ast x) \\
            &= (g/h^{e})(t_{0} \ast x) \\
            &= (g/h^{e})(x)+c(t_{0}) \\
            &= \gamma_{c(t)}\left(t_{0},(g/h^{e})(x)\right) \\
            &= \gamma_{c(t)} \circ (\operatorname{id}_{\mathbb{G}_{a}}, \Phi)(t_{0},x).
    \end{align*}
    So $ \Phi $ is $ \mathbb{G}_{a} $-equivariant.  Since $ \Phi^{\sharp} $ maps $ k[t] $ isomorphically onto $ k[g/h^{e}] $ it is injective.  So $ \Phi $ is dominant.  As a result, b) implies a).

    Now assume that a) holds, i.e., there is a dominant $ \mathbb{G}_{a} $-equivariant morphism \linebreak $ \Phi: D(h) \to \mathbb{G}_{a}^{c(t)} $.  Because $ \Phi $ is $ \mathbb{G}_{a} $-equivariant, the following diagram commutes:
    \begin{equation} \label{E:48}
    \xymatrix{
        \mathbb{G}_{a} \times D(h) \ar[rr]^{\beta} \ar[d]^{(\operatorname{id},\Phi)} & & D(h) \ar[d]^{\Phi} \\
        \mathbb{G}_{a} \times \mathbb{G}_{a}^{c(t)} \ar[rr]^{\gamma_{c(t)}} & & \mathbb{G}_{a}^{c(t)}
        }.
    \end{equation}
    There is some $ g \in A $ and $ e \in \mathbb{N}_{0} $ such that $ \Phi(x) = (g/h^{e})(x) $ for all $ x \in D(h) $.  Since $ \Phi $ is dominant, $ g $ cannot be constant.  If $ x \in D(h) $ and $ t_{0} \in \mathbb{G}_{a} $, then:
    \begin{align*}
        (g/h^{e})(t_{0} \ast x) &= \Phi(t_{0} \ast x) \\
        &= \Phi \circ \beta(t_{0},x) \\
        &= \gamma_{c(t)} \circ (\operatorname{id},\Phi)(t_{0},x) \\
        &= \gamma_{c(t)}\left(t_{0},(g/h^{e})(x)\right) \\
        &= (g/h^{e})(x)+c(t_{0}).
    \end{align*}
    So $ (g,h^{e}) $ is a $ c(t) $-pair.  Therefore, a) implies b).
\end{proof}
\begin{lem} \label{L:connectedNonIdentity}
    If $ \operatorname{Spec}(A) $ is a $ \mathbb{G}_{a} $-variety, $ h \in A^{\mathbb{G}_{a}} $, then $ \Phi: D(h) \to \mathbb{G}_{a}^{b(t)} $ is a dominant, generically smooth, $ \mathbb{G}_{a} $-equivariant morphism whose fibres are connected if and only if $ \Phi $ does not factor through a non-identity, endomorphism $ c(t) $ of $ \mathbb{G}_{a} $ and $ \Phi $ does not contract $ D(h) $ to a point.
\end{lem}
\begin{proof}
    Suppose $ \Phi $ factors through the endomorphism $ t^{p^{i}} $.  In this case, there is some separable, additive polynomial $ d(t) $ such that $ b(t) $ is equal to $ d(t)^{p^{i}} $.  If $ k[w] $ is the affine coordinate ring of $ \mathbb{G}_{a}^{d(t)} $, then the following diagram lists inclusions of fields:
    \begin{equation*}
    \xymatrix{
        \operatorname{Frac}(A) \ar@{-}[d] \\
        k(w) \cong K\left(\mathbb{G}_{a}^{d(t)}\right) \ar@{-}[d] \\
        k(w^{p}) \cong K\left(\mathbb{G}_{a}^{d(t^{p^{i}})}\right) \cong K\left(\mathbb{G}_{a}^{b(t)}\right)
        }.
    \end{equation*}
    The extension $ \operatorname{Frac}(A)/K\left(\mathbb{G}_{a}^{b(t)}\right) $ is not separable.  Therefore, $ \Phi $ is not generically smooth.

    Now suppose there exists a separable, additive polynomial $ c(t) $ such that $ \Phi $ factors through the endomorphism $ c(t) $.  If this is the case, then there is an additive polynomial $ d(t) $ and a $ \mathbb{G}_{a} $-equivariant morphism $ \Phi_{1}: D(h) \to \mathbb{G}_{a}^{d(t)} $ such that $ \Phi $ is equal to $ c(t) \circ \Phi_{1} $.  If $ t_{0} $ is a point of $ \mathbb{G}_{a}^{b(t)} $, then $ \Phi^{-1}(t_{0}) $ is equal to $ \cup_{i=1}^{\deg(c(t))} \Phi_{1}^{-1}(\xi_{i}) $ where $ c(t)-t_{0} = \prod_{i=1}^{\deg(c(t))} (t-\xi_{i}) $.  As a result, the fibres of $ \Phi $ are not connected if $ \Phi $ factors through any such non-identity, $ \mathbb{G}_{a} $-equivariant morphism.  If $ \Phi $ is not dominant, then the image of $ \Phi $ is a dimension zero, connected, sub-variety, i.e., a point.  This would mean that $ \Phi $ contracts $ D(h) $ to a point.

    Now suppose that $ \Phi: D(h) \to \mathbb{G}_{a}^{b(t)} $ is a $ \mathbb{G}_{a} $-equivariant morphism which does not factor through a non-identity endomorphism $ c(t) $ of $ \mathbb{G}_{a} $ and does not contract $ D(h) $ to a point.  Because $ D(h) $ is affine, it embeds as an open sub-variety of a projective variety $ Z $.  We may resolve the indeterminacies of the rational map $ \Phi: Z \dashrightarrow \mathbb{P}^{1}_{k} $ to obtain a morphism $ \Psi: W \to \mathbb{P}^{1}_{k} $.  Here, $ W $ is obtained from $ Z $ via blow-ups and $ \Psi \mid_{D(h)} \cong \Phi $.

    By the Stein factorization theorem, there are morphisms $ c $ and $ \Psi_{1} $, such that: $ c $ is finite, the fibres of $ \Psi_{1} $ are connected, and $ \Psi $ is equal to $ c \circ \Psi_{1} $.  If $ \Psi $ contracts $ W $ to a point, then $ \Phi $ contracts $ D(h) $ to a point.  This contradicts our assumption, so we may assume that $ \Psi $ is surjective and $ \Phi $ is dominant.  Because $ c $ is finite and $ \Psi \mid_{D(h)} \cong \Phi $, there is an additive polynomials $ d(t) $ such that $ c $ is a morphism from $ \mathbb{G}_{a}^{d(t)} $ to $ \mathbb{G}_{a}^{b(t)} $.  By our assumption this morphism is the identity.  As a result, the fibres of $ \Phi $ are connected.

    Suppose that $ \Phi $ is not generically smooth.  There is some $ u $ such that
    \begin{equation*}
        K\left(\mathbb{G}_{a}^{b(t)}\right) \cong k(u).
    \end{equation*}
    If $ \Phi $ is not generically smooth, and $ L $ is the inseparable closure of $ k(u) $ in $ \operatorname{Frac}(A) $, then $ L $ strictly contains $ k(u) $.  Because the transcendence degree of $ k(u) $ is one, the transcendence degree of $ L $ is one.  By L\"{u}roth's theorem $ L $ is isomorphic to $ k(z) $ for some $ z $, and $ u $ is equal to $ z^{p^{i}} $, where the inseparable degree of $ L $ over $ k(u) $ is $ p^{i} $.  So, there is some $ i \in \mathbb{N} $ such that $ \Phi $ factors through an endomorphism $ t^{p^{i}} $.  This is a contradiction of our assumption that $ \Phi $ does not factor through any non-identity, endomorphism $ c(t) $ of $ \mathbb{G}_{a} $.  So $ \Phi $ is generically smooth.  This finishes the proof.
\end{proof}
\begin{prop} \label{P:earlyTrivialBundle}
    If $ \operatorname{Spec}(A) $ is a $ \mathbb{G}_{a} $-variety with action $ \beta $ and $ h \in A^{\mathbb{G}_{a}} $ is a non-zero invariant, then the following statements are equivalent
    \begin{itemize}
        \item[a)] the variety $ D(h) $ is a trivial $ \mathbb{G}_{a} $-bundle over $ D(h)//\mathbb{G}_{a} $,
        \item[b)] there is a dominant, generically smooth, $ \mathbb{G}_{a} $-equivariant morphism $ \Phi: D(h) \to \mathbb{G}_{a} $ such that the fibres of $ \Phi $ are connected,
        \item[c)] there is a principle pair $ (g,h^{e}) $ for some $ g \in A $, and $ e \in \mathbb{N}_{0} $,
        \item[d)] there exists a $ g \in A $ such that if $ \pi: D(h) \to D(h)//\mathbb{G}_{a} $ is the quotient morphism and $ \iota: \mathcal{V}(\langle g \rangle) \cap D(h) \to D(h) $ is the inclusion morphism, then  the stabilizer of any point $ x \in \mathcal{V}(\langle g \rangle) \cap D(h) $ is trivial, and $ \pi \circ \iota $ is an isomorphism.
    \end{itemize}
\end{prop}
\begin{proof}
    If $ D(h) $ is a trivial $ \mathbb{G}_{a} $-bundle, over $ D(h)//\mathbb{G}_{a} $ then there is an isomorphism $ \lambda: D(h) \to \mathbb{G}_{a} \times (D(h)//\mathbb{G}_{a}) $.  If $ p_{1}: \mathbb{G}_{a} \times (D(h)//\mathbb{G}_{a}) \to \mathbb{G}_{a} $ is the projection onto the first component, then we claim $ p_{1} \circ \lambda $ is the desired morphism $ \Phi $.  Because $ A_{h} $ is an integral domain, it has no idempotents.  Therefore $ D(h) $ is connected.  Since the image of $ p_{2} \circ \lambda $ is $ D(h)//\mathbb{G}_{a} $, and the image of a connected sub-variety is also connected, $ D(h)//\mathbb{G}_{a} $ is connected.  Since projections and isomorphisms are generically smooth and dominant, the morphism $ \Phi $ is generically smooth and dominant.  Therefore, a) implies b).

    If b) holds, then let $ Y $ be the scheme theoretic pre-image of zero with its reduced, induced scheme structure.  Because $ \Phi $ has connected fibres, the scheme $ Y $ is irreducible and is therefore a variety.  Let $ \gamma: \mathbb{G}_{a} \times Y \to D(h) $ be the morphism which sends $ (t,y) $ to $ t \ast y $.  We claim that $ \gamma $ is injective.  If there exist $ (t_{1},y_{1}) $ and $ (t_{2},y_{2}) $ in $ \mathbb{G}_{a} \times Y $ such that $ t_{1} \ast y_{1}=t_{2} \ast y_{2} $, then $ y_{1}= (t_{2}-t_{1}) \ast y_{2} $  Since $ Y $ is the scheme theoretic pre-image of zero under $ \Phi $,
    \begin{align*}
        t_{2}-t_{1} &= (t_{2}-t_{1})+ \Phi(y_{2}) \\
        &= \Phi\left((t_{2}-t_{1}) \ast y_{2}\right) \\
        &= \Phi(y_{1}) \\
        &= 0.
    \end{align*}
    Therefore, $ t_{2} $ is equal to $ t_{1} $, whence
    \begin{align*}
        y_{1} &= (t_{2}-t_{1}) \ast y_{2} \\
        &= (t_{2}-t_{2}) \ast y_{2} \\
        &= y_{2}.
    \end{align*}
    So $ \gamma $ is injective.  Let $ \lambda: D(h) \to \mathbb{G}_{a} \times Y $ be the morphism which sends $ x $ to $ (\Phi(x), \Phi(x)^{-1} \ast x) $.  Since $ \gamma \circ \lambda $ is the identity, $ \gamma $ is surjective. Because $ \Phi $ is generically smooth, $ \lambda $ is separable.  As a result, $ \lambda $ is an isomorphism.  Therefore, b) implies a).

    If c) holds, then let $ \Phi: D(h) \to \mathbb{G}_{a} $ be the morphism of varieties which sends a point $ x \in D(h) $ to $ (g/h^{e})(x) $.  We do not yet know if $ \Phi $ is $ \mathbb{G}_{a} $-equivariant or generically smooth.  Because:
    \begin{align*}
        \Phi \circ \beta(t_{0},x) &= \Phi(t_{0} \ast x) \\
        &= (g/h)(t_{0} \ast x) \\
        &= (g/h)(x)+t_{0} \\
        &= \mu_{\mathbb{G}_{a}}\left(t_{0},(g/h)(x)\right) \\
        &= \mu_{\mathbb{G}_{a}} \circ \left(\operatorname{id}_{\mathbb{G}_{a}},\Phi\right)(t_{0},x),
    \end{align*}
    the morphism $ \Phi $ is $ \mathbb{G}_{a} $-equivariant.  Suppose that $ \Phi $ factors through a non-identity endomorphism $ c(t) $.  Let $ c(t) $ be maximal with respect to this property.  If this is the case, then $ \Phi $ is equal to $ c(t) \circ \Phi_{1} $.  There exist $ g_{1} \in A $ and $ e_{1} \in \mathbb{N}_{0} $ such that $ \Phi_{1} $ sends $ x \in D(h) $ to $ (g_{1}/h^{e_{1}})(x) $.  For all $ x \in D(h) $
    \begin{align*}
        c(g_{1}/h^{e_{1}})(x) &= c \circ \Phi_{1}(x) \\
        &= \Phi(x) \\
        &= (g/h^{e})(x),
    \end{align*}
    so $ g/h^{e} $ is equal to $ c(g_{1}/h^{e_{1}}) $.  By Proposition ~\ref{P:specificTrivialBundle} (see page \pageref{P:specificTrivialBundle}) $ (g_{1},h^{e_{1}}) $ is a $ d(t) $-pair for some additive polynomial $ d(t) $.  If $ c(t) $ is not equal to $ t $, then for all $ t_{0} \in \mathbb{G}_{a} $ and $ x \in D(h) $,
    \begin{align*}
        (g/h^{e})(x)+t_{0} &= (g/h^{e})(t_{0} \ast x) \\
        &= c(g_{1}/h_{1}^{e_{1}})(t_{0} \ast x) \\
        &= c((g_{1}/h_{1}^{e_{1}})(t_{0} \ast x)) \\
        &= c((g_{1}/h_{1}^{e_{1}})(x)+d(t_{0})) \\
        &= c((g_{1}/h_{1}^{e_{1}})(x))+c(d(t_{0})) \\
        &= (g/h^{e})(x)+c(d(t_{0})).
    \end{align*}
    So $ c(d(t)) = t $.  This cannot happen for any $ d(t) \in k[t] $ so $ d(t) \notin k[t] $.  This would mean that $ \beta^{\sharp}(g_{1}) \notin A[t] $, which contradicts the fact that $ \beta^{\sharp}(A) \subseteq A[t] $.  So there is no such $ c(t) $.  By Lemma ~\ref{L:connectedNonIdentity} (see page \pageref{L:connectedNonIdentity}), the morphism $ \Phi $ is dominant, generically smooth and the fibres of $ \Phi $ are connected.  Therefore, c) implies b).

    Assume that b) holds.  If $ \Phi: D(h) \to \mathbb{G}_{a} $ is a dominant, generically smooth, $ \mathbb{G}_{a} $-equivariant morphism whose fibres are connected, then there is a $ g \in A $ and $ e \in \mathbb{N}_{0} $ such that $ \Phi $ sends $ x \in D(h) $ to $ (g/h^{e})(x) $.  If $ x \in D(h) $ and $ t_{0} \in \mathbb{G}_{a} $, then because $ \Phi $ is $ \mathbb{G}_{a} $-equivariant:
    \begin{align*}
        (g/h^{e})(t_{0} \ast x) &= \Phi(t_{0} \ast x) \\
        &= \Phi \circ \beta(t_{0},x) \\
        &= \mu_{\mathbb{G}_{a}} \circ (\operatorname{id}_{\mathbb{G}_{a}},\Phi)(t_{0},x) \\
        &= \mu_{\mathbb{G}_{a}}\left(t_{0},(g/h^{e})(x)\right) \\
        &= (g/h^{e})(x)+t_{0}.
    \end{align*}
    As a result, $ (g,h^{e}) $ is a principle pair.  Therefore, b) implies c).  As a result, a), b) and c) are equivalent.

    Assume that c) holds.  Let us denote $ \mathcal{V}(\langle g \rangle) \cap D(h) $ by $ K $.  If $ \Phi: D(h) \to \mathbb{G}_{a} $ is the morphism which sends $ x \in D(h) $ to $ (g/h^{e})(x) $, then $ \Phi^{-1}(0) $ with its reduced induced scheme structure is isomorphic to $ D(h)//\mathbb{G}_{a} $.  Therefore,
    \begin{align*}
        D(h)//\mathbb{G}_{a} & \cong \Phi^{-1}(0) \\
        &\cong K,
    \end{align*}
    and the isomorphism is the composition of $ \iota $ with $ \pi $.  Since $ D(h) $ is a trivial $ \mathbb{G}_{a} $-bundle over its image in $ D(h)//\mathbb{G}_{a} \cong K $, the stabilizer is trivial for any point $ x \in K $.  Therefore b) implies d).

    Assume that d) holds.  Let us denote $ D(h) \cap \mathcal{V}(\langle g \rangle) $ by $ K $ and let $ \Lambda $ be the sub-variety of $ \mathbb{G}_{a} \times K $ whose points are $ (t_{0},y) $ such that $ t_{0} \in (\mathbb{G}_{a})_{y} $.  If $ \iota $ is the inclusion of $ K $ in $ D(h) $ and $ \pi: D(h) \to D(h)//\mathbb{G}_{a} $ is the quotient morphism, then let $ \phi $ be the inverse of $ \pi \circ \iota $.  If $ y \in D(h) $, and $ x $ is equal to $ \phi \circ \pi(y) $, then $ \pi(x)=\pi(y) $.  If $ p_{2} $ is the projection of $ \mathbb{G}_{a} \times K $ onto $ K $, then the following diagram commutes:
    \begin{equation*}
    \xymatrix{
        \mathbb{G}_{a} \times K \ar[rr]^{\beta} \ar[drr]^{\pi \circ p_{2}} & & D(h) \ar[d]^{\pi} \\
        & & D(h)//\mathbb{G}_{a}
        }.
    \end{equation*}
    Therefore, $ \beta: \mathbb{G}_{a} \times K \to D(h) $ is surjective.  Also note that $ \beta \mid_{\Lambda} = p_{2} \mid_{\Lambda} $.  There may not be a good quotient $ (\mathbb{G}_{a} \times D(h))/\Lambda $ in the category of schemes, however there is one in the category of algebraic spaces.  Let $ \nu: \mathbb{G}_{a} \times K \to (\mathbb{G}_{a} \times K)/\Lambda $ be the quotient morphism.  By the universal property of the co-equalizer there is a morphism $ \psi $, such that the following diagram commutes:
    \begin{equation*}
    \xymatrix{
        \mathbb{G}_{a} \times K \ar[dr]^{\nu} \ar[r]^{\beta} & D(h) \\
        & (\mathbb{G}_{a} \times K)/\Lambda \ar[u]^{\psi}
        }.
    \end{equation*}
    If $ y \in D(h) $, and if $ x $ is equal to $ \pi \circ \iota(y) $, then let $ (t_{0},x),(t_{1},x) $ be two points such that
    \begin{align*}
        t_{0} \ast x &= y \\
        &=t_{1} \ast x.
    \end{align*}
    If $ q_{2}: \Lambda \to K $, then by our assumption that the stabilizer of any point $ x \in K $ is trivial,
    \begin{align*}
        (t_{0}-t_{1},x) &\in \Lambda \mid_{x} \\
        &= (0,x).
    \end{align*}
    So $ \beta: \mathbb{G}_{a} \times K \to D(h) $ is injective, surjective and separable.  Therefore
    \begin{align*}
        D(h) & \cong \mathbb{G}_{a} \times K \\
        & \cong \mathbb{G}_{a} \times (D(h)//\mathbb{G}_{a}).
    \end{align*}
    So a) holds. 
\end{proof}
The reader will be familiar with Rosenlicht's cross section theorem (see \cite{RosenlichtCS}).  Item d) relates Rosenlicht's cross section theorem with principle pairs.  One can see that Rosenlicht's cross section theorem is weaker than the condition that $ D(h) $ is a principle $ \mathbb{G}_{a} $-bundle over its image in the quotient.  For an example see Example ~\eqref{EG:1} (page \pageref{EG:1}).
\begin{rmk}
    Let us talk about why the condition that Rosenlicht's cross section theorem is not equivalent to the condition that there is an open $ \mathbb{G}_{a} $-equivariant sub-variety $ D(h) $ such that $ D(h) $ is a trivial $ \mathbb{G}_{a} $-variety over $ D(h)//\mathbb{G}_{a} $.  Let $ k $ be an algebraically closed field of characteristic $ p>0 $ and let $ c_{1}(t),c_{2}(t) $ be two $ k $-linearly independent additive polynomials.  Let $ \mathbf{V} $ be a three dimensional vector space with dual basis $ \{x_{1},x_{2},x_{3}\} $ and let $ \beta: \mathbb{G}_{a} \to \operatorname{GL}(\mathbf{V}) $ be the linear representation with the following co-action:
    \begin{align*}
        x_{1} & \mapsto x_{1} \\
        x_{2} & \mapsto x_{2} \\
        x_{3} & \mapsto x_{3}+c_{2}(t)x_{2}+c_{1}(t)x_{1}.
    \end{align*}
    If $ y \in \mathbf{V} $ is a point equal to $ (a_{1},a_{2},a_{3}) $, then let $ w(y) $ be the point of $ \mathbb{G}_{a} $ such that
    \begin{align*}
        a_{3} = c_{2}(w(y))a_{2}+c_{1}(w(y))a_{1}.
    \end{align*}
    The morphism $ \lambda: \mathbf{V} \to \mathbb{G}_{a} \times \operatorname{Spec}(k[x_{1},x_{2}]) $ which sends $ y $ to $ (w(y),a_{1},a_{2}) $ is an isomorphism of varieties.  However, we claim that it is not $ \mathbb{G}_{a} $-equivariant on any open sub-variety $ D(h) $.  This will mean that $ \mathbf{V} $ is not a trivial $ \mathbb{G}_{a} $-variety.  Let $ h \in k[x_{1},x_{2}] $ and suppose that $ \lambda: D(h) \to \mathbb{G}_{a} \times \operatorname{Spec}(k[x_{1},x_{2}]_{h}) $ is $ \mathbb{G}_{a} $-equivariant.
    
    If $ \lambda $ is $ \mathbb{G}_{a} $-equivariant, then the following diagram commutes:
    \begin{equation*}
    \xymatrix{
        \mathbb{G}_{a} \times D(h) \ar[rr]^{(\operatorname{id},\lambda)} \ar[d]^{\beta} & & \mathbb{G}_{a} \times \mathbb{G}_{a} \times \operatorname{Spec}(k[x_{1},x_{2}]_{h}) \ar[d]^{(\mu_{\mathbb{G}_{a}}, \operatorname{id})} \\
        D(h) \ar[rr]^{\lambda} & & \mathbb{G}_{a} \times \operatorname{Spec}(k[x_{1},x_{2}]_{h}) 
        }.
    \end{equation*}
    So, if $ \lambda $ is $ \mathbb{G}_{a} $-equivariant, $ t_{0} \in \mathbb{G}_{a} $ and $ y \in D(h) $ is equal to $ (a_{1},a_{2},a_{3}) $, then:
    \begin{align*}
        (w(t_{0} \ast y), a_{1},a_{2}) &= \lambda(a_{1},a_{2},a_{3}+c_{2}(t_{0})a_{2}+c_{1}(t_{0})a_{1}) \\
        &= \lambda \circ \beta(t_{0},y) \\
        &= (\mu_{\mathbb{G}_{a}}, \operatorname{id}) \circ (\operatorname{id},\lambda)(t_{0},y) \\
        &= (\mu_{\mathbb{G}_{a}}, \operatorname{id})(t_{0},(w(y),a_{1},a_{2})) \\
        &= (w(y)+t_{0},a_{1},a_{2}).
    \end{align*}
    Because
    \begin{equation}
        w(t_{0} \ast y) = w(y)+t_{0},
    \end{equation}
    must hold, $ (w,1) $ is a principle pair.  The rational function $ w $ is equal to $ (g(X),h^{e}) $ for some polynomial $ g(X) $ and $ e \in \mathbb{N}_{0} $.  Therefore $ (g(X),h^{e}) $ is a principle pair.  We will prove later that pairs do not exist at all for this representation.  As a result, the morphism $ \lambda $ is an isomorphism of varieties, but not $ \mathbb{G}_{a} $-varieties.  What occurs in this example is that the stabilizers $ (\mathbb{G}_{a})_{y} $ vary wildly for $ y \in \mathbf{V} $.  If $ \iota: \mathcal{V}(\langle x_{3} \rangle) \to \mathbf{V} $ is the inclusion morphism and $ \pi: \mathbf{V} \to \operatorname{Spec}(k[x_{1},x_{2}]) $ is the quotient morphism, then $ \pi \circ \iota $ is an isomorphism.  However, there is no open sub-variety $ U \subseteq \mathbf{V} $ with the property that $ (\mathbb{G}_{a})_{y} $ is trivial for all $ y \in U $.  As a result, there is no open, affine, $ \mathbb{G}_{a} $-stable, sub-variety $ D(h) $ such that $ D(h) $ is a trivial $ \mathbb{G}_{a} $-bundle over $ D(h)//\mathbb{G}_{a} $.
\end{rmk}
\begin{prop}
    Let $ \operatorname{Spec}(A) $ be a $ \mathbb{G}_{a} $-variety with action $ \beta $, and let $ h \in A^{\mathbb{G}_{a}} $.  The following are equivalent:
    \begin{itemize}
        \item[a)] there is an $ \mathbb{G}_{a} $-equivariant, fppf neighborhood $ \phi:U \to D(h) $ such that $ \phi^{-1}(x) \cong \mathbf{ker}(c(t)) $ for some additive polynomial $ c(t) \in \mathfrak{O} $ and such that $ U $ is a trivial $ \mathbb{G}_{a} $-bundle over $ U//\mathbb{G}_{a} $,
        \item[b)] there is a $ g \in A $ and $ e \in \mathbb{N}_{0} $ such that $ (g,h^{e}) $ is a $ c(t) $-pair.
    \end{itemize}
\end{prop}
\begin{proof}
    Assume that b) holds.  Let $ \mathbb{G}_{a} $ acts on $ \mathbb{A}^{1}_{k} \times D(h) $ via the action $ \gamma $ which sends $ (t_{0},(s,x)) $ to $ (s+t_{0}, t_{0} \ast x) $.  If $ k[w] $ is the coordinate ring of $ \mathbb{A}^{1}_{k} $, then let $ U $ equal $ \mathcal{V}(\langle c(w)-g/h^{e} \rangle) $.  The ideal $ c(w)-g/h^{e} $ is stable under the action of $ \gamma $ because:
    \begin{align*}
        \gamma^{\sharp}(c(w)-g/h^{e}) &= c(w+t)-\beta^{\sharp}(g/h^{e}) \\
        &= c(w)+c(t)-(g/h^{e}+c(t)) \\
        &= c(w)-g/h^{e}.
    \end{align*}
    The pair $ (w,1) $ is a principle pair on $ U $.  So $ U \cong \mathbb{G}_{a} \times (U//\mathbb{G}_{a}) $ by Proposition ~\ref{P:earlyTrivialBundle}.  Let $ \phi: U \to D(h) $ be the natural projection morphism.  If $ x \in D(h) $, and $ (t_{1},x),(t_{2},x) \in \phi^{-1}(x) $, then
    \begin{align*}
        c(t_{1})-(g/h^{e})(x) &= 0 \\
        &= c(t_{2})-(g/h^{e})(x).
    \end{align*}
    So $ c(t_{1}-t_{2}) $ is equal to zero.  As a result, $ \phi^{-1}(x) \cong \mathbf{ker}(c(t)) $.  Let $ Z $ be the closure of $ U $ in $ \mathbb{P}^{1}_{D(h)} $.  For any $ x \in D(h) $, the fibre $ Z_{x} $ is equal to the closure of $ \mathbf{ker}(c(t)) $ in $ \mathbb{P}^{1}_{k} $.  Therefore, the fibres of $ Z_{x} $ all have the same Hilbert polynomial.  So $ Z $ is flat over $ D(h) $.  Because $ U $ is the complement of the point at infinity in $ Z $, the morphism $ \phi $ from $ U $ to $ D(h) $ is flat.  Since a flat, local morphism is faithfully flat, $ \phi: U \to D(h) $ is an fppf morphism.  So a) holds.
    
    Assume that a) holds.  If $ \lambda: U \to \mathbb{G}_{a} \times (U//\mathbb{G}_{a}) $ is the $ \mathbb{G}_{a} $-bundle isomorphism and if $ p_{1} $ is the projection from $ \mathbb{G}_{a} \times (U//\mathbb{G}_{a}) $ onto the first component, then the morphism $ p_{1} \circ \lambda $ is a $ \mathbb{G}_{a} $-equivariant, dominant, morphism from $ U $ to $ \mathbb{G}_{a} $.  Let us denote $ p_{1} \circ \lambda $ by $ \Psi $.  If $ x \in D(h) $ and $ y_{1},y_{2} \in \phi^{-1}(x) $, then because $ \phi^{-1}(x) \cong \mathbf{ker}(c(t)) $ there is a point $ z \in U//\mathbb{G}_{a} $ and points $ t_{1},t_{2} \in \mathbb{G}_{a} $ such that
    \begin{align*}
        \lambda(y_{1}) &= (t_{1},z) \\
        \lambda(y_{2}) &= (t_{2},z),
    \end{align*}
    and $ t_{1}-t_{2} \in \mathbf{ker}(c(t)) $.  So if $ \Phi $ is the morphism which sends $ x \in D(h) $ to $ \Psi(y) $ for any $ y \in \phi^{-1}(x) $, then $ \Phi $ does not depend on the choice of $ y \in \phi^{-1}(x) $.  The following diagram commutes:
    \begin{equation*}
    \xymatrix{
        U \ar[rr]^{\Psi} \ar[d]^{\phi} & & \mathbb{G}_{a} \ar[d]^{c(t)} \\
        D(h) \ar[rr]^{\Phi} & & (\mathbb{G}_{a})^{c(t)}
        }.
    \end{equation*} 
    So because $ \Psi, \phi $ and $ c(t) $ are $ \mathbb{G}_{a} $-equivariant, dominant morphisms, $ \Phi $ is also $ \mathbb{G}_{a} $-equivariant and dominant.  By Proposition ~\ref{P:specificTrivialBundle} b) holds.
\end{proof}
\begin{prop} \label{P:trivialBundle}
    If $ \operatorname{Spec}(A) $ is a $ \mathbb{G}_{a} $-variety with action $ \beta $, and $ D(h) $ is a $ \mathbb{G}_{a} $-stable, open, sub-variety, then the following statements are equivalent:
    \begin{itemize}
        \item[a)] the variety $ D(h) \cong (\mathbb{G}_{a}//\mathbf{ker}(b(t)) \times D(h)//\mathbb{G}_{a}) $,
        \item[b)] there is a dominant, generically smooth, $ \mathbb{G}_{a} $-equivariant, morphism \linebreak $ \Psi: D(h) \to \mathbb{G}_{a}^{b(t)} $ such that the fibres are connected and $ \mathbf{ker}(b(t)) $ acts trivially upon $ \operatorname{Spec}(A) $.
    \end{itemize}
\end{prop}
\begin{proof}
    If statement a) holds, then we claim that $ \mathbf{ker}(b(t)) $ acts trivially on $ D(h) $.  Because $ \mathbb{G}_{a} $ acts trivially on $ D(h)//\mathbb{G}_{a} $ and $ \mathbf{ker}(b(t)) $ acts trivially upon $ \mathbb{G}_{a}//\mathbf{ker}(b(t)) $, it acts trivially upon $ \left(\mathbb{G}_{a}//\mathbf{ker}(b(t)) \times D(h)//\mathbb{G}_{a}\right) \cong D(h) $.  Therefore, we may replace $ \mathbb{G}_{a} $ with $ \mathbb{G}_{a}//\mathbf{ker}(b(t)) $.  If we do this, then we may assume that $ \mathbb{G}_{a} $ acts freely upon $ D(h) $ and that $ D(h) $ is isomorphic to $ \mathbb{G}_{a} \times D(h)//\mathbb{G}_{a} $, i.e., that $ D(h) $ is a trivial $ \mathbb{G}_{a} $-bundle.

    If statement b) holds, then $ \mathbf{ker}(b(t)) $ acts trivially on $ D(h) $.  If we replace $ \mathbb{G}_{a} $ with $ \mathbb{G}_{a}//\mathbf{ker}(b(t)) $ then it suffices to show that the following two statements are equivalent:
    \begin{itemize}
        \item[aa)] the variety $ D(h) $ is a trivial $ \mathbb{G}_{a} $-bundle,
        \item[bb)] there is a generically smooth, dominant, $ \mathbb{G}_{a} $-equivariant, morphism $ \Psi: D(h) \to \mathbb{G}_{a} $ whose fibres are connected.
    \end{itemize}

    This now follows from Proposition ~\ref{P:earlyTrivialBundle} (see page \pageref{P:earlyTrivialBundle}).
\end{proof}
\begin{dfn} \label{D:freudenbergIdeal}
    Let $ \operatorname{Spec}(A) $ be a $ \mathbb{G}_{a} $-variety over an algebraically closed field $ k $ of characteristic $ p>0 $.  Let $ S $ be the union of zero and the set of additive polynomials $ c(t) $ such that there is a non-trivial $ c(t) $-pair $ (g,h) $.  The \emph{fundamental ideal} is the left ideal of the Ore ring generated by the set $ S $.  We denote the fundamental ideal by $ \mathfrak{f}_{A} $ or by $ \mathfrak{f} $ when the action is clear.  Over an algebraically closed field \cite{Ore} shows that a left ideal of the Ore ring is right principle.  Hence there is a generator $ b(t) $ of $ \mathfrak{f} $.
\end{dfn}
By Lemma ~\ref{L:pairsLemma}, if $ b(t) $ generates $ \mathfrak{f}_{A} $, then there is a $ b(t) $-pair $ (g,h) $.
\begin{rmk}
    The reader may note that we included the statement ``$ \dots $ and $ \mathbf{ker}(b(t)) $ acts trivially upon $ \operatorname{Spec}(A) $,'' in Proposition ~\ref{P:trivialBundle} (see page \pageref{P:trivialBundle}).  They may ask why this is needed, and what one obtains without this statement.  If $ \mathbb{G}_{a} $ acts on $ \operatorname{Spec}(A) $, then the statement that there exists a dominant, generically smooth, $ \mathbb{G}_{a} $-equivariant morphism $ \Phi: D(h) \to (\mathbb{A}^{1}_{k})^{b(t)} $ is equivalent to the condition that $ \mathfrak{f}= \mathfrak{O} \langle b(t) \rangle $.  However, there is no guarantee a'priori that $ \mathbf{ker}(b(t)) $ stabilizes $ \operatorname{Spec}(A) $.  In ~\eqref{E:89} we see an example of a representation where $ \mathfrak{f}= \mathfrak{O} \langle b(t) \rangle $, but $ \mathbf{ker}(b(t)) $ does not stabilize $ \mathbf{V} $.
\end{rmk}
\begin{lem} \label{L:pairsLemma}
    Let $ \operatorname{Spec}(A) $ be a $ \mathbb{G}_{a} $-variety with action $ \beta $.  If $ (g_{1},h_{1}) $ is a $ c_{1}(t) $-pair, $ (g_{2},h_{2}) $ is a $ c_{2}(t) $-pair, and $ b(t) $ is equal to $ \mathfrak{O}(c_{1}(t),c_{2}(t)) $, then there is a $ b(t) $-pair $ (g,h) $.  If $ A $ is a UFD, then $ h $ is the product of elements of the set containing the irreducible factors of $ h_{1} $ and $ h_{2} $.
\end{lem}
\begin{proof}
    Let $ b_{1}(t),b_{2}(t) $ and $ g/h $ be defined by the relations below:
    \begin{align*}
        b(t)&=b_{1}(c_{1}(t))+b_{2}(c_{2}(t)) \\
        g/h &= b_{1}(g_{1}/h_{1})+b_{2}(g_{2}/h_{2}) \\
        \gcd(g,h) &= 1.
    \end{align*}
    If $ t_{0} \in \mathbb{G}_{a} $ and $ x \in D(h) $, then the following computations now show that $ (g,h) $ is a $ b(t) $-pair
    \begin{align*}
        (g/h)(t_{0} \ast x) &= \left(b_{1}(g_{1}/h_{1})+b_{2}(g_{2}/h_{2})\right)(t_{0} \ast x) \\
        &=b_{1}(g_{1}/h_{1})(t_{0} \ast x)+b_{2}(g_{2}/h_{2})(t_{0} \ast x) \\
        &= b_{1}((g_{1}/h_{1})(t_{0} \ast x))+b_{2}(g_{2}/h_{2}(t_{0} \ast x) \\
        &= b_{1}((g_{1}/h_{1})(x)+c_{1}(t_{0}))+b_{2}((g_{2}/h_{2})(x)+c_{2}(t_{0})) \\
        &= b_{1}((g_{1}/h_{1})(x))+b_{1}(c_{1}(t_{0}))+b_{2}((g_{2}/h_{2})(x))+b_{2}(c_{2}(t_{0})) \\
        &= b_{1}((g_{1}/h_{1})(x))+b_{2}((g_{2}/h_{2})(x))+b_{1}(c_{1}(t_{0}))+b_{2}(c_{2}(t_{0})) \\
        &= b_{1}((g_{1}/h_{1})(x))+b_{2}((g_{2}/h_{2})(x))+b(t_{0}) \\
        &= b_{1}(g_{1}/h_{1})(x)+b_{2}(g_{2}/h_{2})(x)+b(t_{0}) \\
        &= (g/h)(x)+b(t_{0}).
    \end{align*}
    Therefore, $ (g,h) $ is a $ b(t) $-pair.  If $ A $ is a UFD, then $ h $ is the product of elements of the set containing the irreducible factors of $ h_{1} $ and $ h_{2} $ because
    \begin{equation*}
        g/h= b_{1}(g_{1}/h_{1})+b_{2}(g_{2}/h_{2}).
    \end{equation*}
\end{proof}
\begin{cor} \label{Cor:fundamental}
    Let $ \operatorname{Spec}(A) $ be a $ \mathbb{G}_{a} $-variety with action $ \beta $.  If $ \mathfrak{f} $ is equal to $ \mathfrak{O} \langle b(t) \rangle $, then there is a $ b(t) $-pair $ (g,h) $.
\end{cor}
We shall recall the statement of Luna's \'{E}tale slice theorem,
\begin{quote}
    Soit $ G $ un groupe r\'{e}ductif qui op\`{e}re dans une vari\'{e}t\'{e} affine $ X $.  Soit $ x $ un point de $ X $, dont l'orbite $ G(x) $ est ferm\'{e}e.

    Il existe alors une sous-vari\'{e}t\'{e} $ V $ de $ X $ qui les propri\'{e}t\'{e}s suivantes: elle est affine et contient $ x $; le groupe d'isotropie $ G_{x} $ laisse $ V $ stable; l'op\'{e}ration de $ G $ dans $ X $ induit un $ G $-morphisme \'{e}tale $ \psi: G \times_{G_{x}} V \to X $; l'image $ U $ de $ \psi $ est un ouvert, affine, $ \pi_{X} $ satur\'{e} de $ X $; le morphisme $ \psi/G: (G \times_{G_{x}} V)/G \cong V/G_{x} \to U/G $ est \'{e}tale; enfin; le morphismes $ \psi $ et $ G \times_{G_{x}} V \to (G \times_{G_{x}} V)/G \cong V/G_{x} $ induisent un $ G $-isomorphisme $ G \times_{G_{x}} V \cong U \times_{U/G} (V/G_{x}) $.

    Translation:

    Let $ G $ be a reductive group which acts on an affine variety $ X $.  Let $ x $ be a point of $ X $ whose orbit $ G(x) $ is closed.

    There exists a sub-variety $ V $ of $ X $ with the following properties: it is affine and contains $ x $; $ V $ is stable under the action of the isotropy group $ G_{x} $; the action of $ G $ on $ X $ induces an \'{e}tale $ G $-equivariant morphism $ \psi: G \times_{G_{x}} V \to X $; the image $ U $ of $ \psi $ is an open, affine, dense sub-variety of $ X $; the morphism $ \psi/G: (G \times_{G_{x}} V)/G \cong V/G_{x} \to U/G $ is \'{e}tale; finally; the morphisms $ \psi $ and $ G \times_{G_{x}} V \to (G \times_{G_{x}} V)/G \cong V/G_{x} $ induce a $ G $-equivariant isomorphism $ G \times_{G_{x}} V \cong U \times_{U/G} (V/G_{x}) $.
\end{quote}
Recall that if $ \operatorname{Spec}(A) $ is a $ \mathbb{G}_{a} $-variety, and $ H $ is a sub-group scheme of $ \mathbb{G}_{a} $, then $ H $ acts on $ \mathbb{G}_{a} \times \operatorname{Spec}(A) $ via the action which sends $ (h,(t_{0},x)) $ to $ (t_{0}+h, h^{-1} \ast x) $.  We write $ \mathbb{G}_{a} \times_{H} \operatorname{Spec}(A) $ to denote the quotient of $ \mathbb{G}_{a} \times \operatorname{Spec}(A) $ by this action.
\begin{prop} \label{P:principlePairBundle}
    Let $ \operatorname{Spec}(A) $ be a $ \mathbb{G}_{a} $-variety and $ h \in A^{\mathbb{G}_{a}} $.  The following conditions are equivalent
    \begin{itemize}
        \item[a)] there is a dominant, generically smooth, $ \mathbb{G}_{a} $-equivariant, morphism \linebreak $ \Phi: D(h) \to \left(\mathbb{A}^{1}_{k}\right)^{b(t)} $ such that the fibres of $ \Phi $ are connected, and $ \mathbf{ker}(b(t)) $ acts trivially upon $ D(h) $,
        \item[b)] there is a quasi-principle $ b(t) $-pair $ (g,h^{e}) $ for some additive polynomial $ b(t) $, $ g \in A $ and $ e \in \mathbb{N}_{0} $,
        \item[c)] there exists a sub-variety $ Y $ of $ D(h) $ such that
            \begin{itemize}
                \item[i)] it is affine,
                \item[ii)] it is fixed by $ \mathbf{ker}(b(t)) $,
                \item[iii)] the action of $ \mathbb{G}_{a} $ on $ D(h) $ induces an \'{e}tale $ \mathbb{G}_{a} $-equivariant morphism $ \psi: \mathbb{G}_{a} \times_{\mathbf{ker}(b(t))} Y \cong D(h) $,
                \item[iv)] the morphism $ \psi/\mathbb{G}_{a}: (\mathbb{G}_{a} \times_{\mathbf{ker}(b(t))} Y)/\mathbb{G}_{a} \cong Y \to D(h)/\mathbb{G}_{a} $ is \'{e}tale,
                \item[v)] the morphisms $ \psi $ and $ \mathbb{G}_{a} \times_{\mathbf{ker}(b(t))} Y \to (D(h) \times_{D(h)//\mathbb{G}_{a})} Y)/\mathbb{G}_{a} \cong Y $ induce a $ \mathbb{G}_{a} $-equivariant isomorphism
                    \begin{equation*}
                        \left(\mathbb{G}_{a} \times_{\mathbf{ker}(b(t))} Y\right) \cong \left(D(h) \times_{D(h)/\mathbb{G}_{a}} Y\right).
                    \end{equation*}
            \end{itemize}
    \end{itemize}
\end{prop}
\begin{proof}
    If a) holds, then there is clearly a $ b(t) $-pair $ (g,h^{e}) $ for some $ g \in A $ and $ e \in \mathbb{N}_{0} $ by Proposition ~\ref{P:specificTrivialBundle} (see page \pageref{P:specificTrivialBundle}).  Since $ \mathbf{ker}(b(t)) $ acts trivially upon $ D(h) $, the pair $ (g,h^{e}) $ is a quasi-principle pair.  Therefore, a) implies b).

    If b) holds, then by Proposition ~\ref{P:specificTrivialBundle} (see page \pageref{P:specificTrivialBundle}), there is a $ \mathbb{G}_{a} $-equivariant morphism $ \Phi: D(h) \to \mathbb{G}_{a}^{b(t)} $.  If $ \Phi $ factors through a non-identity endomorphism $ c(t) $ of $ \mathbb{G}_{a} $, then let $ c(t) $ be maximal with respect to this property.  If this is the case, then there is an additive polynomial $ d(t) $ and a dominant, $ \mathbb{G}_{a} $-equivariant morphism $ \Phi_{1}: D(h) \to \mathbb{G}_{a}^{d(t)} $ such that $ c(t) \circ \Phi_{1} = \Phi $.  There is a $ g_{1} \in A $ and $ e_{1} \in \mathbb{N}_{0} $ such that $ \Phi_{1} $ maps $ x \in D(h) $ to $ (g_{1}/h^{e_{1}})(x) $.  By Proposition ~\ref{P:specificTrivialBundle} (see page \pageref{P:specificTrivialBundle}) $ (g_{1},h^{e_{1}}) $ is a $ d(t) $-pair.  Because $ \Phi= c \circ \Phi_{1} $,
    \begin{equation*}
        c(d(t)) =b(t).
    \end{equation*}
    Since $ \beta^{\sharp} $ factors through $ \operatorname{id}_{A} \otimes b(t)^{\sharp} $, by Lemma ~\ref{L:principlePairEquiv}
    \begin{align*}
        \beta^{\sharp}(g_{1}) &= g_{1}+d(t)h^{e_{1}} \\
        &\in A[b(t)].
    \end{align*}
    So $ c(t) $ must equal $ t $, contrary to our assumption.  Therefore $ \Phi $ is dominant, generically smooth and the fibres of $ \Phi $ are connected.  So, b) implies a).

    Assume that b) holds.  We proved in Proposition ~\ref{P:specificTrivialBundle} (see page \pageref{P:specificTrivialBundle}) that a $ b(t) $-pair $ (g,h^{e}) $ exists and in Lemma ~\ref{L:principlePairEquiv} (see page \pageref{L:principlePairEquiv}) that $ \mathbf{ker}(b(t)) $ is the largest, additive, sub group-scheme which acts trivially on $ \operatorname{Spec}(A) $.  Proposition ~\ref{P:trivialBundle} (see page \pageref{P:trivialBundle}) now shows that if $ Y $ is $ \Phi^{-1}(0) $ with its reduced induced scheme structure, then $ D(h) \cong \left(Y \times_{\mathbf{ker}(b(t))} \mathbb{G}_{a} \right) $.  Therefore, ii) and iii) hold where $ \psi $ is the action obtained of $ \mathbb{G}_{a} \cong \mathbb{G}_{a}/\mathbf{ker}(b(t)) $ on $ Y \times_{\mathbf{ker}(b(t))} \mathbb{G}_{a} $.  The variety $ Y $ is clearly a categorical quotient of $ D(h) $ by $ \mathbb{G}_{a} $.  Because $ \operatorname{Spec}(A_{h}^{\mathbb{G}_{a}}) $ is a categorical quotient,
    \begin{align*}
        Y &\cong D(h)//\mathbb{G}_{a} \\
        & \cong \operatorname{Spec}(A_{h}^{\mathbb{G}_{a}}).
    \end{align*}
    Therefore, i) holds.  The morphism $ \psi//\mathbb{G}_{a} $ is an isomorphism; therefore, iv) holds.  Condition v) holds because:
    \begin{align*}
        D(h) & \cong \left(Y \times_{\mathbf{ker}(b(t))} \mathbb{G}_{a}\right) \\
        & \cong \left(D(h) \times_{D(h)//\mathbb{G}_{a}} Y\right) \\
        &\cong \left(D(h) \times_{Y} Y\right),
    \end{align*}
    Therefore, b) implies c).

    If c) holds, then $ D(h) \cong \left(Y \times_{\mathbf{ker}(b(t))} \mathbb{G}_{a}\right) $ for a variety $ Y $ with a trivial $ \mathbf{ker}(b(t)) $-action.  The variety $ Y $ is a categorical quotient, so by the uniqueness of categorical quotients $ Y \cong D(h)//\mathbb{G}_{a} $.  By Proposition ~\ref{P:trivialBundle} (see page \pageref{P:trivialBundle}) there is a dominant, generically smooth, $ \mathbb{G}_{a} $-equivariant, morphism $ \Phi: D(h) \to (\mathbb{G}_{a})^{b(t)} $.  So there exists a $ b(t) $-pair by Proposition ~\ref{P:specificTrivialBundle} (see page \pageref{P:specificTrivialBundle}).  Such a pair is quasi-principle by definition.
\end{proof}
\begin{dfn} \label{D:plinth1}
    If $ \operatorname{Spec}(A) $ is a $ \mathbb{G}_{a} $-variety, then the \emph{pedestal ideal} is the ideal of $ A $ generated by $ \{ 0 $ and all $ h $ such that there exists some non-zero, additive polynomial $ b(t) $ and a $ g \in A $ such that $ (g,h) $ is a quasi-principle $ b(t) $-pair $ \} $ (see Definition ~\ref{D:aPair} on page \pageref{D:aPair}).  Denote this ideal by $ \mathfrak{P}(A) $ or just $ \mathfrak{P} $ when the underlying ring is clear.  The \emph{pedestal scheme} is $ \mathcal{V}(\mathfrak{P}(A)) \subseteq \operatorname{Spec}(A) $.  A point $ x \in \operatorname{Spec}(A) $ is \emph{affine stable} if $ x \in \operatorname{Spec}(A) \setminus \mathcal{V}(\mathfrak{P}(A)) $.  If $ Z $ is an arbitrary (not necessarily affine) variety, then $ x \in Z $ is affine stable if there is a $ \mathbb{G}_{a} $-stable, affine, neighborhood $ \operatorname{Spec}(A) $ such that $ x $ is an affine stable point of $ \operatorname{Spec}(A) $.  We denote the sub-variety of affine stable points of $ \operatorname{Spec}(A) $ by $ \operatorname{Spec}(A)^{as} $.  The variety $ \operatorname{Spec}(A) $ is quasi-principle if there is a quasi-principle pair $ (g,h) $.  An arbitrary variety $ Z $ is quasi-principle if $ Z^{as} $ is non-empty.

    If $ \operatorname{Spec}(A) $ is a $ \mathbb{G}_{a} $-variety, then the \emph{large pedestal ideal} is the ideal of $ A $ generated by $ \{ h $ such that there exists a non-zero, additive polynomial $ c(t) $ and a $ g \in A $ such that $ (g,h) $ is a $ c(t) $-pair $ \} $.  We denote the large pedestal ideal of $ A $ by $ \mathfrak{P}_{g}(A) $.
\end{dfn}
If $ \operatorname{Spec}(A) $ is a quasi-principle variety with action $ \beta $, and $ (g,h) $ is a quasi-principle $ b(t) $-pair, then $ \mathbf{ker}(b(t)) $ acts trivially on $ \operatorname{Spec}(A) $.  Since
\begin{align*}
    A^{\mathbb{G}_{a}} &= (A^{\mathbf{ker}(b(t))})^{(\mathbb{G}_{a}//\mathbf{ker}(b(t)))} \\
    &= A^{\mathbb{G}_{a}//\mathbf{ker}(b(t))},
\end{align*}
it often makes sense to replace $ \mathbb{G}_{a} $ by $ \mathbb{G}_{a}//\mathbf{ker}(b(t)) $ and assume that $ (g,h) $ is a principle pair.  One may notice that by Proposition ~\ref{P:earlyTrivialBundle} (see page \pageref{P:earlyTrivialBundle}) the definition that $ \operatorname{Spec}(A) $ is a generically principle variety is equivalent to the existence of a principle pair.
    
The pedestal ideal is related to the concept of the plinth ideal.  The plinth ideal has been studied in characteristic zero in the situation where a $ \mathbb{G}_{a} $-action is determined by a locally nilpotent derivation.  If the ring $ A^{\mathbb{G}_{a}} $ is a finitely generated ring over a field of characteristic zero, then the plinth ideal is equal to $ \mathfrak{P}(A) \cap A^{\mathbb{G}_{a}} $.
\begin{rmk}
    The pedestal ideal is non-empty, since it contains zero.  The large pedestal ideal is non-empty since for $ h \in A^{\mathbb{G}_{a}} $, the pair $ (h,0) $ is a $ c(t) $-pair for any non-zero, additive polynomial $ c(t) $.
\end{rmk}
The reader should avoid the pitfall of believing that all representations have a non-zero large pedestal ideal.  In example ~\ref{EG:1} (see page \pageref{EG:1}) we provide an example of three dimensional, linear representation $ \beta: \mathbb{G}_{a} \to \operatorname{GL}(\mathbf{V}) $ such that $ \mathfrak{P}_{g}(S_{k}(\mathbf{V}^{\ast})) $ is equal to zero.  Moreover, we will give a strict criterion classifying all representations such that $ \mathfrak{P}_{g}(S_{k}(\mathbf{V}^{\ast})) $ is equal to zero.
\section{When the Large Pedestal Ideal is Equal to Zero.}
\begin{lem} \label{L:deltaLemma}
    If $ \beta $ is an action of $ \mathbb{G}_{a} $ on $ \operatorname{Spec}(A) $ and we denote $ \beta^{\sharp}(f) -f $ by $ \delta(f) $, then $ \delta(fg)=\delta(f)\delta(g)+f\delta(g)+\delta(f)g $.  Also, $ \delta $ is an $ A^{\mathbb{G}_{a}} $-module homomorphism from $ A $ to $ A[t] $; i.e., if $ h \in A^{\mathbb{G}_{a}} $ and $ g \in A $, then $ \delta(gh) = h\delta(g) $.
\end{lem}
\begin{proof}
    Observe that
    \begin{equation} \label{E:34}
        \delta(fg) =\beta^{\sharp} (fg)-fg.
    \end{equation}
    If we add and subtract $ f\beta^{\sharp}(g) $ from ~\eqref{E:34}, then
    \begin{equation} \label{E:35}
        \delta(fg)= \delta(f)\beta^{\sharp}(g)+f \delta(g).
    \end{equation}
    By adding and subtracting $ \delta(f)g $ from the right hand side of ~\eqref{E:35}, we obtain the desired result.  The last part of the Lemma is clear.
\end{proof}
\begin{cor} \label{Cor:expansion}
    If $ \beta $ is an action of $ \mathbb{G}_{a} $ on $ \mathbb{A}^{n}_{k} \cong \operatorname{Spec}(k[X]) $, then $ \delta(x_{j}^{i}) $ is equal to $ \sum_{\ell=1}^{i} \binom{i}{\ell} \delta(x_{j})^{\ell} x_{j}^{i-\ell} $.  If $ x_{j} $ is not invariant, then $ \delta(x_{j}^{i})= \delta(x_{j})^{i} $ if and only if $ i $ is a power of $ p $.
\end{cor}
\begin{lem} \label{L:invtMult}
    Let $ \operatorname{Spec}(A) $ be a $ \mathbb{G}_{a} $-variety with action $ \beta $ and no non-trivial units.  If $ f,g \in A \setminus \{0\} $ and $ fg \in A^{\mathbb{G}_{a}} $, then $ f $ and $ g $ are invariant.
\end{lem}
\begin{proof}
    Let $ \rho_{x} $ denote right translation by $ x \in \mathbb{G}_{a} $.  If $ f,g \in A $ and $ fg \in A^{\mathbb{G}_{a}} $, then
    \begin{align*}
        \rho_{x}^{\sharp}(f)\rho_{x}^{\sharp}(g) &= \rho_{x}^{\sharp}(fg) \\
        &= fg.
    \end{align*}
    As a result, $ \left(\rho_{x}^{\sharp}(f)/f\right)\left(\rho_{x}^{\sharp}(g)/g\right) = 1 $.  Because there are no non-trivial units in $ A $,
    \begin{align*}
        \rho_{x}^{\sharp}(f)/f &= c_{x} \\
        & \in k^{\ast}.
    \end{align*}
    If $ x,y \in \mathbb{G}_{a} $, then
    \begin{align*}
        c_{x+y} &= \rho_{x+y}^{\sharp}(f)/f \\
        &= \left(\rho_{x+y}^{\sharp}(f)/\rho_{y}^{\sharp}(f)\right) \left(\rho_{y}^{\sharp}(f)/f\right) \\
        &= \rho_{y}^{\sharp}\left(\rho_{x}^{\sharp}(f)/f\right) \left(\rho_{y}^{\sharp}(f)/f\right) \\
        &= \rho_{y}^{\sharp}(c_{x}) c_{y} \\
        &= c_{x} c_{y}.
    \end{align*}
    As a result, the function which sends $ x $ to $ c_{x} $ is a character of $ \mathbb{G}_{a} $.  So $ c_{x} = 1 $ for all $ x \in \mathbb{G}_{a} $, because $ \mathbb{G}_{a} $ has no non-trivial characters.  As a result, $ f,g \in A^{\mathbb{G}_{a}} $.
\end{proof}
\begin{lem} \label{L:beginningEx}
    Let $ \beta: \mathbb{G}_{a} \to \operatorname{GL}(\mathbf{V}) $ be an $ n $-dimensional, linear representation, and let $ \{x_{1},\dots,x_{n-1}\} $ be a basis of the socle of $ \mathbf{V}^{\ast} $ which extends to a basis $ \{x_{1},\dots,x_{n}\} $ of $ \mathbf{V}^{\ast} $.  The ring of invariants is equal to $ k[x_{1},\dots,x_{n-1}] $.
\end{lem}
\begin{proof}
    Assume that $ f(X) $ is an invariant polynomial.  There exist polynomials $ \{f_{j}(x_{1},\dots,x_{n-1})\}_{j=0}^{d} $ such that $ f(X) $ is equal to $ \sum_{j=0}^{d} f_{j}(x_{1},\dots,x_{n-1})x_{n}^{j} $.  If $ g(X) $ is equal to
    \begin{equation*}
        f(X)-f_{0}(x_{1},\dots,x_{n-1}),
    \end{equation*}
    then $ g(X) $ is also invariant.  Let $ g(X) $ equal $ x_{n} g_{1}(X) $.  If $ g_{1}(X) $ is non-zero, then $ x_{n} $ and $ g_{1}(X) $ are invariant by Lemma ~\ref{L:invtMult} (see page \pageref{L:invtMult}).  However, $ x_{n} $ is not invariant.  This is a contradiction unless
    \begin{align*}
        g(X) &= g_{1}(X) \\
        &= 0.
    \end{align*}
    Since $ g(X) $ is equal to zero, $ f(X) \in k[x_{1},\dots,x_{n-1}] $.
\end{proof}
\begin{eg} \label{EG:1}
    The following is an example of a three-dimensional, indecomposable, representation such that the large pedestal ideal is equal to zero.  Let $ k $ be a field of characteristic $ p>0 $, and let $ c_{1}(t) $ and $ c_{2}(t) $ be two linearly independent, additive, polynomials.  Also let $ \{u_{1},u_{2},u_{3}\} $ be a basis of a three-dimensional vector space $ \mathbf{V} $ with dual basis $ \{x_{1},x_{2},x_{3}\} $.  Define $ \beta: \mathbb{G}_{a} \to \operatorname{GL}(\mathbf{V}) $ to be the three-dimensional representation with the following co-action:
    \begin{align*}
        x_{1} & \mapsto x_{1} \\
        x_{2} & \mapsto x_{2} \\
        x_{3} & \mapsto x_{3}+c_{2}(t)x_{2}+c_{1}(t)x_{1}.
    \end{align*}
    This is only possible if the characteristic of $ k $ is greater than zero.  We shall now show that $ \mathbf{V} $ is indecomposable.

    The co-action for $ \beta^{\vee} $ is described below:
    \begin{align*}
        u_{3} & \mapsto u_{3} \\
        u_{2} & \mapsto u_{2}+c_{2}(t)u_{3} \\
        u_{1} & \mapsto u_{1}+c_{1}(t)u_{3}.
    \end{align*}
    If $ b_{1}u_{1}+b_{2}u_{2}+b_{3}u_{3} $ is invariant, then:
    \begin{align*}
        b_{1}u_{1}+b_{2}u_{2}+b_{3}u_{3} & = (\beta^{\vee})^{\sharp}(b_{1}u_{1}+b_{2}u_{2}+b_{3}u_{3}) \\
        &= b_{1}(\beta^{\vee})^{\sharp}(u_{1})+b_{2}(\beta^{\vee})^{\sharp}(u_{2})+b_{3} (\beta^{\vee})^{\sharp}(u_{3}) \\
        &= b_{1}(u_{1}+c_{1}(t)u_{3})+b_{2}(u_{2}+c_{2}(t)u_{3})+b_{3}u_{3} \\
        &= b_{1}u_{1}+b_{2}u_{2}+b_{3}u_{3}+(b_{1}c_{1}(t)+b_{2}c_{2}(t))u_{3}.
    \end{align*}
    The polynomial $ b_{1}c_{1}(t)+b_{2}c_{2}(t) $ must equal zero for $ b_{1}u_{1}+b_{2}u_{2}+b_{3}u_{3} $ to be invariant.  This contradicts the fact that $ c_{1}(t) $ and $ c_{2}(t) $ are linearly independent, unless both $ b_{1} $ and $ b_{2} $ are equal to zero.

    So the dimension of $ \mathbf{V}^{\mathbb{G}_{a}} $ is one.  Therefore, $ \mathbf{V} $ and $ \mathbf{V}^{\ast} $ are indecomposable.

    The ring of invariants of $ k[x_{1},x_{2},x_{3}] $ is $ k[x_{1},x_{2}] $ by Lemma ~\ref{L:beginningEx}.  Assume that there is a $ b(t) $-pair $ (g(X),h(X)) $.  There are polynomials $ \{g_{j}(x_{1},x_{2})\}_{j=0}^{d} $ such that $ g(X) $ is equal to $ \sum_{j=0}^{d} g_{j}(x_{1},x_{2})x_{3}^{j} $.  Let $ S $ be the set below:
    \begin{equation*}
        S = \{ j \, \, \text{s.t.} \, \, g_{j}(x_{1},x_{2}) \ne 0 \, \, \text{and} \, \, j \, \, \text{is not a power of } p\}.
    \end{equation*}  
    If $ G_{1}(X) $ and $ G_{2}(X) $ are described below:
    \begin{align*}
        G_{1}(X)&:= \sum_{j \in S} g_{j}(x_{1},x_{2})x_{3}^{j} \\
        G_{2}(X)&:= \sum_{j \notin S} g_{j}(x_{1},x_{2})x_{3}^{j},
    \end{align*}
    then $ g(X) $ is equal to $ G_{1}(X)+G_{2}(X) $.  By Corollary ~\ref{Cor:expansion} (see page \pageref{Cor:expansion}),
    \begin{equation*}
        \delta(x_{3}^{j})-(\delta(x_{3})^{j}) \in \langle x_{3} \rangle k[X][t],
    \end{equation*}
    so
    \begin{align*}
        k[x_{1},x_{2}][t] & \ni h(X)b(t)-\delta(G_{2}(X))-G_{1}(x_{1},x_{2},\delta(x_{3})) \\
        &= \delta(G_{1}(X))-G_{1}(x_{1},x_{2},\delta(x_{3})) \\
        &= \delta(\sum_{j\in S} g_{j}(x_{1},x_{2})x_{3}^{j})-\sum_{j\in S^{c}} g_{j}(x_{1},x_{2})\delta(x_{3})^{j} \\
        &= \sum_{j \in S} g_{j}(x_{1},x_{2})\delta(x_{3}^{j})-\sum_{j \in S^{c}} g_{j}(x_{1},x_{2})\delta(x_{3})^{j} \\
        &= \sum_{j \in S} g_{j}(x_{1},x_{2})\left(\delta(x_{3}^{j})-\delta(x_{3})^{j}\right) \\
        & \in \langle x_{3} \rangle k[x_{1},x_{2},x_{3}][t].
    \end{align*}
    Therefore, $ \delta(G_{1}(X)) =G_{1}(x_{1},x_{2},\delta(x_{3})) $.  However, a polynomial $ H(X) $ equal to $ \sum_{j=0}^{s} h_{j}(x_{1},x_{2})x_{3}^{j} $ has the property that $ \delta(H(X)) = H(x_{1},x_{2},\delta(x_{3})) $ if and only if $ h_{j}(x_{1},x_{2}) $ is zero whenever $ j $ is not a power of $ p $ by Corollary ~\ref{Cor:expansion}.  So $ G_{1}(X) $ is equal to zero.

    By our assumption that $ G_{2}(X) $ is the sum of polynomials of the form $ g_{p^{j}}(x_{1},x_{2})x_{3}^{p^{j}} $ for $ j \in \mathbb{N}_{0} $, there are polynomials $ \{g_{j}(x_{1},x_{2})\}_{j=0}^{s} $ such that $ g(X) $ is equal to $ \sum_{j=0}^{s} g_{j}(x_{1},x_{2})x_{3}^{p^{j}} $.  By Lemma ~\ref{L:deltaLemma} (see page \pageref{L:deltaLemma}) and its Corollary,
    \begin{align}
        \delta(g(X))&= \delta\left(\sum_{j=0}^{s} g_{j}(x_{1},x_{2})x_{3}^{p^{j}}\right), \notag \\
        &= \sum_{j=0}^{s} g_{j}(x_{1},x_{2})\delta(x_{3}^{p^{j}}), \notag \\
        &= \sum_{j=0}^{s} g_{j}(x_{1},x_{2})\delta(x_{3})^{p^{j}}, \notag \\
        &= \sum_{j=0}^{s} g_{j}(x_{1},x_{2}) \left(c_{1}(t)x_{1}+c_{2}(t)x_{2}\right)^{p^{j}}, \notag \\
        &= \sum_{j=0}^{s} g_{j}(x_{1},x_{2}) (c_{1}(t)x_{1})^{p^{j}}+ \sum_{j=0}^{s} g_{j}(x_{1},x_{2})(c_{2}(t)x_{2})^{p^{j}}, \notag \\
        &= g(x_{1},x_{2},c_{1}(t)x_{1})+g(x_{1},x_{2},c_{2}(t)x_{2}). \label{E:59}
    \end{align}
    Equation ~\eqref{E:59} shows that the variance of $ g(X) $ is greater than or equal to three (see Definition ~\ref{D:pseudoPair} on page \pageref{D:pseudoPair}).  This contradicts the fact that $ (g(X),h(X)) $ is a $ b(t) $-pair by Corollary ~\ref{Cor:pairSubSpace} (see page \pageref{Cor:pairSubSpace}).  As a result, the large pedestal ideal of $ k[x_{1},x_{2},x_{3}] $ is equal to zero.

    We will show that if $ \mathbf{V} $ is a linear representation such that $ \mathfrak{P}_{g}(S_{k}(\mathbf{V}^{\ast})) $ is equal to zero, then $ S_{k}(\mathbf{V}^{\ast})^{\mathbb{G}_{a}} $ is equal to $ S_{k}\left((\mathbf{V}^{\ast})^{\mathbb{G}_{a}}\right) $.  As a result, a representation whose large pedestal ideal is equal to zero is utterly uninteresting from the perspective of classical invariant theory.
\end{eg}
\begin{thm} \label{T:zeroLargePlinth}
    If $ \beta: \mathbb{G}_{a} \to \operatorname{GL}(\mathbf{V}) $ is an $ n $-dimensional, linear representation, such that $ \mathfrak{P}_{g}(S_{k}(\mathbf{V}^{\ast})) $ is equal to zero, then $ S_{k}(\mathbf{V}^{\ast})^{\mathbb{G}_{a}} $ is equal to $ S_{k}\left((\mathbf{V}^{\ast})^{\mathbb{G}_{a}}\right) $.
\end{thm}
\begin{proof}
    Let us induce on the dimension $ n $ of $ \mathbf{V} $.  If $ n $ is equal to one, then the representation is trivial, and the theorem holds.  Now assume that if $ n<N $ and $ \beta: \mathbb{G}_{a} \to \operatorname{GL}(\mathbf{V}) $ is an $ n $-dimensional, linear representation such that $ \mathfrak{P}_{g}(S_{k}(\mathbf{V}^{\ast})) $ is equal to zero, then $ S_{k}(\mathbf{V}^{\ast})^{\mathbb{G}_{a}} $ is equal to $ S_{k}\left((\mathbf{V}^{\ast})^{\mathbb{G}_{a}}\right) $.

    Let $ \beta:\mathbb{G}_{a} \to \operatorname{GL}(\mathbf{V}) $ be an $ N $-dimensional, linear representation such that $ \mathfrak{P}_{g}(S_{k}(\mathbf{V}^{\ast})) $ is equal to zero.  Let $ \{x_{1},\dots,x_{m}\} $ be a basis of the socle of $ \mathbf{V}^{\ast} $ and extend it to an upper triangular basis $ \{x_{1},\dots,x_{N}\} $ of $ \mathbf{V}^{\ast} $.  If $ \{x_{1},\dots,x_{N-1}\} $ is a dual basis for $ \mathbf{W} $, then the induction hypothesis holds for $ \mathbf{W} $.

    If $ \overline{x_{N}} $ is equal to $ x_{N}+ \langle x_{1},\dots,x_{N-1} \rangle $, then let us denote $ k[x_{1},\dots,x_{N-1}] $ by $ k[\widetilde{X}] $ and a polynomial of $ k[\widetilde{X}] $ by $ h(\widetilde{X}) $.  Let $ M $ be the infinite dimensional vector space spanned by $ \cup_{j=1}^{\infty} \{\overline{x_{N}}^{j}\} $.  The following sequence of vector spaces is exact:
    \begin{equation} \label{E:54}
    \xymatrix{
        0 \ar[r] & k[\widetilde{X}] \ar[r] & k[X] \ar[r] & M \ar[r] & 0
        }.
    \end{equation}
    Because $ k[\widetilde{X}]^{\mathbb{G}_{a}} $ is equal to $ k[x_{1},\dots,x_{m}] $ the following diagram commutes:
    \begin{equation*}
    \xymatrix{
          & 0 \ar[d] & 0 \ar@{-->}[d] & 0 \ar[d] \\
        0 \ar[r] & k[x_{1},\dots,x_{m}] \ar[d]  \ar[r] & k[X]^{\mathbb{G}_{a}} \ar@{-->}[d] \ar[r] & M \ar[d] \\
        0 \ar[r] & k[x_{1},\dots,x_{m}] \ar[r] & k[x_{1},\dots,x_{m}, x_{N}] \ar[r] & M
        },
    \end{equation*}
    and so $ k[X]^{\mathbb{G}_{a}} $ is a sub-space of $ k[x_{1},\dots,x_{m},x_{N}] $ by the five lemma \cite[Chapter 2, Hom and Tensor, Section 2, Tensor Products, Proposition 2.72, (ii)]{Rotman}.  If \linebreak $ f(X) \in k[X]^{\mathbb{G}_{a}} $, then we may write $ f(X) $ as follows:
    \begin{equation*}
        f(X) = \sum_{j=0}^{d} r_{j}(x_{1},\dots,x_{m}) x_{N}^{j},
    \end{equation*}
    for some $ d \in \mathbb{N}_{0} $ and polynomials $ \{r_{j}(x_{1},\dots,x_{m})\}_{j=0}^{d} $.  Since $ f(X) \in k[X]^{\mathbb{G}_{a}} $,
    \begin{equation*}
        f(X)-r_{0}(x_{1},\dots,x_{m}) \in k[X]^{\mathbb{G}_{a}}.
    \end{equation*}
    Let $ f_{1}(X) $ equal $ \sum_{j=1}^{d} r_{j}(x_{1},\dots,x_{m})x_{N}^{j-1} $.  If $ f_{1}(X) $ is non-zero then
    \begin{align}
        x_{N}f_{1}(X) &= f(X)-r_{0}(x_{1},\dots,x_{m}) \notag \\
        &\in k[X]^{\mathbb{G}_{a}}. \label{E:58}
    \end{align}
    However, by Lemma ~\ref{L:invtMult} (see page \pageref{L:invtMult}) the only way ~\eqref{E:58} can hold is if $ f_{1}(X) $ is equal to zero.  So,
    \begin{align*}
        f(X) &= r_{0}(x_{1},\dots,x_{m}) \\
        & \in k[x_{1},\dots,x_{m}].
    \end{align*}
\end{proof}
If $ \beta: \mathbb{G}_{a} \to \operatorname{GL}(\mathbf{V}) $ is an $ n $-dimensional, linear representation, $ \{x_{1},\dots,x_{n}\} $ is an upper triangular basis of $ \mathbf{V}^{\ast} $ and we write $ \beta^{\sharp}(x_{i}) $ as $ v_{i}(X,t) $, then the graph morphism $ \gamma: \mathbb{G}_{a} \times \mathbf{V} \to \mathbf{V} \times \mathbf{V} $ is equal to $ (p_{2},\beta) $ where $ p_{2} $ is the natural projection of $ \mathbb{G}_{a} \times \mathbf{V} $ onto $ \mathbf{V} $.  There is a $ \mathbb{G}_{a} $-equivariant embedding of $ \mathbf{V} $ into
\begin{align*}
    \mathbb{P}^{n}_{k} & \cong \mathbb{P}(k \oplus \mathbf{V}^{\ast}) \\
    & \cong \operatorname{Proj}(k[z_{0},\dots,z_{n}]),
\end{align*}
and a $ \mathbb{G}_{a} $-equivariant embedding of $ \mathbb{G}_{a} $ into $ \mathbb{P}^{1}_{k} \cong \operatorname{Proj}(k[t_{0},t_{1}]) $, which identifies $ \mathbb{G}_{a} $ with $ D_{+}(t_{1}) $.  It may not be possible to extend $ \gamma $ to a morphism $ \lambda $ from $ \mathbb{P}^{1}_{k} \times \mathbb{P}^{n}_{k} $ to $ \mathbb{P}^{n}_{k} \times \mathbb{P}^{n}_{k} $ such that $ \lambda \mid_{D_{+}(t_{1}) \times D_{+}(z_{0})} = \gamma $.  However, there may be an open sub-variety $ U \subseteq \mathbf{V} $ such that there exists $ \lambda: \mathbb{P}^{1}_{k} \times U \to U \times \mathbb{P}^{n}_{k} $ with the property that $ \lambda \mid_{D_{+}(t_{1}) \times U} = \gamma \mid_{\mathbb{G}_{a} \times U} $.  If there is such an open sub-variety, then let $ \Gamma $ be the closure of the image of $ \lambda $ in $ \mathbb{P}^{n}_{k} \times \mathbb{P}^{n}_{k} $.  If $ k[W,Y] $ is the homogeneous coordinate ring of $ \mathbb{P}^{n}_{k} \times \mathbb{P}^{n}_{k} $, then one might suspect that if $ \{ f_{i}(W,Y)\}_{i=1}^{\ell} $ is a reduced, Groebner basis of the ideal sheaf of $ \Gamma $ with respect to the co-lexicographical ordering and $ \sum_{J} f_{i,J}(W)Y^{J} $ is the expansion of $ f_{i}(W,Y) $, then the set $ \{f_{i,J}(1,x_{1},\dots,x_{n})\}_{1\le i \le \ell, J} $ might separate orbits on $ U $.  The next theorem shows that there is such an open sub-variety $ U $ and that this set of polynomials does indeed separate orbits on $ U $.
\begin{thm} \label{T:graphSep}
    Let $ \beta: \mathbb{G}_{a} \to \operatorname{GL}(\mathbf{V}) $ be an $ n $-dimensional, linear representation of $ \mathbb{G}_{a} $ and let $ \{x_{1},\dots,x_{n}\} $ be an upper triangular basis of $ \mathbf{V}^{\ast} $.  There is an open sub-variety $ U \subseteq \mathbf{V} $ and a morphism $ \lambda: \mathbb{P}^{1}_{k} \times U \to U \times \mathbb{P}^{n}_{k} $ such that $ \lambda \mid_{D_{+}(t_{1}) \times U} $ is equal to $ \gamma \mid_{\mathbb{G}_{a} \times U} $.  Let $ \Gamma $ be the closure of the image of $ \lambda $ in $ \mathbb{P}^{n}_{k} \times \mathbb{P}^{n}_{k} $, and let us denote the homogeneous coordinate ring of $ \mathbb{P}^{n}_{k} \times \mathbb{P}^{n}_{k} $ by $ k[W][Y] $.
    If $ \{f_{i}(W,Y)\}_{i=1}^{\ell} $ is a reduced Groebner basis of the ideal sheaf of $ \Gamma $ in $ \mathbb{P}^{n}_{k} \times \mathbb{P}^{n}_{k} $ with respect to the co-lexicographical ordering on $ k[W][Y] $, i.e.
    \begin{equation*}
        w_{0} \prec_{\operatorname{colex}} \cdots \prec_{\operatorname{colex}} w_{n} \prec_{\operatorname{colex}} y_{0} \prec_{\operatorname{colex}} \cdots \prec_{\operatorname{colex}} y_{n},
    \end{equation*}
    and
    \begin{equation*}
        f_{i}(W,Y) = \sum_{J} f_{i,J}(W)Y^{J},
    \end{equation*}
    with respect to multi-index notation, then the set of invariant functions below:
    \begin{equation*}
        \{f_{i,J}(X)\}_{i=1,\dots,\ell, J \, \, \text{s.t. } f_{i,J}(X) \ne 0}
    \end{equation*}
    separates orbits on $ U $.
\end{thm}
\begin{proof}
    If $ \beta^{\sharp}(x_{i}) $ is equal to $ v_{i}(X,t) $, then let $ d_{i} $ equal $ \deg_{t}(v_{i}(X,t)) $.  If $ d $ is equal to the maximum of $ d_{1},\dots,d_{n} $, then let $ q_{i}(Z,t_{0},t_{1}) $ equal $ v_{i}(z_{1},\dots,z_{n},t_{0}/t_{1})t_{1}^{d} $.  Let $ \{\phi_{i}\}_{i \in \mathbb{N}_{0}} $ be an appropriate collection of endomorphisms such that $ v_{i}(X,t) $ is equal to $ \sum_{j=0}^{d_{i}} \phi_{j}(x_{i}) t^{j} $.  Also let $ S $ be the set of $ i $ such that $ \deg_{t}(v_{i}(X,t))$ is equal to $ d $, and let $ U $ be the open sub-variety $ \left(\mathbf{V} \setminus \mathcal{V}(\langle \phi_{d_{i}}(x_{i}) \rangle_{i \in S})\right) $.  We may define $ \lambda^{\sharp}: k[W][Y] \to k[Z][t_{0},t_{1}] $ to be the ring homomorphism below:
    \begin{align*}
        \lambda^{\sharp}(y_{0}) &= z_{0}t_{1}^{d} \\
        \lambda^{\sharp}(y_{i}) &= q_{i}(Z,t_{0},t_{1}) \quad 1 \le i \le n \\
        \lambda^{\sharp}(w_{i}) &= z_{i} \quad 0 \le i \le n.
    \end{align*}
    We obtain a morphism of varieties $ \lambda: \mathbb{P}^{1}_{k} \times U \to U \times \mathbb{P}^{n}_{k} $ from the morphism $ \lambda^{\sharp} $.  By construction $ \lambda \mid_{D_{+}(t_{1}) \times U} = \gamma \mid_{\mathbb{G}_{a} \times U} $.

    Suppose that $ a=(a_{1},\dots,a_{n}) $ and $ b=(b_{1},\dots,b_{n}) $ are two points of $ U $ such that $ f_{i,J}(a) = f_{i,J}(b) $ for $ 1 \le i \le \ell $ and all $ J $.  Because $ a = 0 \ast a $, the following identities hold:
    \begin{align*}
        0 &= f_{i}(a,a) \\
        &= \sum_{J} f_{i,J}(a) \prod_{s=1}^{n} a_{s}^{j_{s}} \\
        &= \sum_{J} f_{i,J}(b) \prod_{s=1}^{n} a_{s}^{j_{s}} \\
        &= f_{i}(b,a).
    \end{align*}
    As a result, $ (b,a) \in \Gamma_{U} $.  Therefore, $ b \in \mathbb{G}_{a} \ast a $.

    If $ a \in U $ and $ t_{0} \in \mathbb{G}_{a} $, then
    \begin{align}
        \sum_{J} f_{i,J}(a) \prod_{s=1}^{n} a_{s}^{j_{s}} &= f_{i}(a,a), \notag \\
        &=0, \notag \\
        &= f_{i}(t_{0} \ast a,a), \notag \\
        &= \sum_{J} f_{i,J}(t_{0} \ast a) \prod_{s=1}^{n} a_{s}^{j_{s}}. \label{E:6}
    \end{align}
    The left hand side of ~\eqref{E:6} does not depend on $ t_{0} $ while the right hand side would depend on $ t_{0} $ if $ f_{i,J}(X) $ is not invariant.  Therefore $ f_{i,J}(X) $ is invariant.
\end{proof}
\begin{thm} \label{T:trivialLargePedForm}
    If $ \beta: \mathbb{G}_{a} \to \operatorname{GL}(\mathbf{V}) $ is a non-trivial, representation, then the large pedestal ideal of $ S_{k}(\mathbf{V}^{\ast}) $ is equal to zero if and only if: the dimension of $ \operatorname{soc}_{2}(\mathbf{V}^{\ast})/(\mathbf{V}^{\ast})^{\mathbb{G}_{a}} $ is equal to one, the length of the socle series of $ \mathbf{V}^{\ast} $ is two and the variance of any element of $ \operatorname{soc}_{2}(\mathbf{V}^{\ast}) \setminus (\mathbf{V}^{\ast})^{\mathbb{G}_{a}} $ is greater than two.  Moreover, if the large pedestal ideal of $ S_{k}(\mathbf{V}^{\ast}) $ is equal to zero, then the characteristic of $ k $ is greater than zero.
\end{thm}
\begin{proof}
    Assume that $ \mathfrak{P}_{g}(S_{k}(\mathbf{V}^{\ast})) $ is equal to zero.  By Theorem ~\ref{T:graphSep} (see page \pageref{T:graphSep}), it is possible to separate orbits using invariant functions on a dense open sub-variety of $ \mathbf{V} $.  By Theorem ~\ref{T:zeroLargePlinth} (see page \pageref{T:zeroLargePlinth}), the ring $ S_{k}(\mathbf{V}^{\ast})^{\mathbb{G}_{a}} $ is equal to $ S_{k}\left((\mathbf{V}^{\ast})^{\mathbb{G}_{a}}\right) $.  Let $ \{x_{1},\dots,x_{m}\} $ be a basis of $ (\mathbf{V}^{\ast})^{\mathbb{G}_{a}} $, and extend it to a basis $ \{x_{1},\dots,x_{n}\} $ of $ \mathbf{V}^{\ast} $.  If $ n>m+1 $, then the fibers of the morphism $ \mathbf{V} \to \operatorname{Spec}\left(S_{k}\left((\mathbf{V}^{\ast})\right)^{\mathbb{G}_{a}}\right) $ are at least two-dimensional.  So it is impossible to separate orbits, contradicting Theorem ~\ref{T:graphSep} (see page \pageref{T:graphSep}).  As a result, $ n $ is equal to $ m+1 $.  Moreover the variance of $ x_{m+1} $ is at least three, or else the large pedestal ideal is non-zero by Corollary ~\ref{Cor:pairSubSpace} (see page \pageref{Cor:pairSubSpace}).  It is also impossible for the variance of $ x_{m+1} $ to be greater than two if the characteristic of $ k $ is equal to zero.  Therefore, the characteristic of $ k $ is $ p>0 $ if the large pedestal ideal of $ k[X] $ is equal to zero.

    Assume that the dimension of $ \operatorname{soc}_{2}\left(\mathbf{V}^{\ast}/(\mathbf{V}^{\ast})\right)^{\mathbb{G}_{a}} $ is equal to one, the length of the socle series of $ \mathbf{V}^{\ast} $ is two, the characteristic of $ k $ is $ p>0 $, and the variance of any element of $ \operatorname{soc}_{2}(\mathbf{V}^{\ast})\setminus (\mathbf{V}^{\ast})^{\mathbb{G}_{a}} $ is greater than two.  If $ \mathbf{V}^{\ast} $ is $ n $-dimensional, then we may extend a basis $ \{x_{1},\dots,x_{n-1}\} $ of $ (\mathbf{V}^{\ast})^{\mathbb{G}_{a}} $ to a basis $ \{x_{1},\dots,x_{n}\} $ of $ \mathbf{V}^{\ast} $.  If $ \beta^{\sharp}(x_{n}) $ is equal to $ x_{n}+\sum_{i=1}^{n-1} c_{i}(t)x_{i} $, then
    \begin{align*}
        x_{n}+\sum_{i=1}^{n-1} c_{i}(t \otimes 1+1\otimes t)x_{i} &= (\operatorname{id}_{k[X]} \otimes \mu_{\mathbb{G}_{a}}^{\sharp})\left(x_{n}+\sum_{i=1}^{n-1} c_{i}(t)x_{i}\right) \\
         &= (\operatorname{id}_{k[X]} \otimes \mu_{\mathbb{G}_{a}}^{\sharp}) \circ \beta^{\sharp}(x_{n}) \\
         &= (\beta^{\sharp} \otimes \operatorname{id}_{k[t]}) \circ \beta^{\sharp}(x_{n}) \\
         &= (\beta^{\sharp} \otimes \operatorname{id}_{k[t]}) \left(x_{n}+\sum_{i=1}^{n-1} c_{i}(1 \otimes t) \otimes x_{i}\right) \\
         &= x_{n}+\sum_{i=1}^{n-1}\left(c_{i}(t \otimes 1)+c_{i}(1 \otimes t)\right)x_{i}.
    \end{align*}
    As a result, each $ c_{i}(t) $ is additive for $ 1 \le i \le n-1 $.  Also, note that $ k[X]^{\mathbb{G}_{a}} $ is equal to $ k[x_{1},\dots,x_{n-1}] $ by Lemma ~\ref{L:beginningEx} (see page \pageref{L:beginningEx}).  Let us denote $ k[x_{1},\dots,x_{n-1}] $ by $ k[\widetilde{X}] $ and an element $ h(X) $ of $ k[\widetilde{X}] $ by $ h(\widetilde{X}) $.

    If the large pedestal ideal of $ k[X] $ is not equal to zero, then there is a non-trivial $ c(t) $-pair $ (g(X),h(\widetilde{X})) $ for some additive polynomial $ c(t) $.  If $ g(X) $ is equal to $ \sum_{j=0}^{d} g_{j}(\widetilde{X})x_{n}^{j} $, then let $ S $ be the set below:
    \begin{equation*}
        S = \{ j \, \, \text{s.t.} \, \, g_{j}(\widetilde{X}) \ne 0 \, \, \text{and} \, \, j \, \, \text{is not a power of } p\}.
    \end{equation*} 
    Let $ G_{1}(X) $ and $ G_{2}(X) $ be the polynomials described below:
    \begin{align*}
        G_{1}(X) &:= \sum_{j \in S} g_{j}(\widetilde{X})x_{n}^{j} \\
        G_{2}(X) &:= \sum_{j \notin S} g_{j}(\widetilde{X})x_{n}^{j}.
    \end{align*}
    Because $ (g(X),h(\widetilde{X})) $ is a $ c(t) $-pair,
    \begin{align}
        \delta(G_{1}(X))+\delta(G_{2}(X))&= \delta(g(X)), \notag \\
        &= h(\widetilde{X})c(t). \label{E:57}
    \end{align}
    Since $ \delta(x_{n}^{j})-\delta(x_{n})^{j} \in \langle x_{n} \rangle k[X][t] $,
    \begin{align*}
        \langle x_{n} \rangle k[X][t] & \ni \sum_{j\in S} g_{j}(\widetilde{X})\left(\delta(x_{n}^{j})-\delta(x_{n})^{j}\right) \\
        &= \left(\sum_{j \in S} g_{j}(\widetilde{X})\delta(x_{n}^{j})\right)-\left(\sum_{j \in S} g_{j}(\widetilde{X})\delta(x_{n})^{j}\right) \\
        &=\delta(G_{1}(X))-G_{1}(\widetilde{X},\delta(x_{n})) \\
        &=h(\widetilde{X})c(t)-\delta(G_{2}(X))-G_{1}(\widetilde{X},\delta(x_{n})) \\
        &\in k[\widetilde{X}][t],
    \end{align*}
    it must be the case that $ G_{1}(X) $ is equal to $ G_{1}(\widetilde{X},\delta(x_{n})) $. 
    
    However, by Corollary ~\ref{Cor:expansion} (see page \pageref{Cor:expansion}) if $ H(X) $ is equal to $ \sum_{j=0}^{s} h_{j}(\widetilde{X})x_{n}^{j} $ and $ \delta(H(X)) = H(\widetilde{X},\delta(x_{n})) $, then $ h_{j}(\widetilde{X}) $ is equal to zero unless $ j $ is a power of $ p $.  So $ G_{1}(X) $ is equal to zero.  If $ \beta^{\sharp}(x_{n}) $ is equal to $ x_{n}+\sum_{j=1}^{n-1} c_{j}(t)x_{j} $, then $ c_{j}(t) $ is additive by our earlier calculations.  If $ G_{2}(X) $ is equal to $ \sum_{j=0}^{s} g_{j}(\widetilde{X}) x_{n}^{p^{j}} $, then
    \begin{align}
        \delta(g(X)) &= \delta(G_{2}(X)), \notag \\
        &= \delta(\sum_{j=0}^{s} g_{j}(\widetilde{X}) x_{n}^{p^{j}}), \notag \\
        &= \sum_{j=0}^{s} g_{j}(\widetilde{X}) \delta(x_{n}^{p^{j}}), \notag \\
        &= \sum_{j=0}^{s} g_{j}(\widetilde{X}) \delta(x_{n})^{p^{j}}, \notag \\
        &= \sum_{j=0}^{s} g_{j}(\widetilde{X}) \sum_{i=1}^{n-1} c_{i}(t)^{p^{j}} x_{i}^{p^{j}}, \notag \\
        &= \sum_{j=0}^{s} g_{j}(\widetilde{X}) \sum_{i=1}^{n-1} c_{i}(t)^{p^{j}} x_{i}^{p^{j}}, \notag \\
        &= \sum_{i=1}^{n-1} \sum_{j=0}^{s} g_{j}(\widetilde{X})\left(c_{i}(t)x_{i}\right)^{p^{j}}, \notag \\
        &= \sum_{i=1}^{n-1} G_{2}(\widetilde{X},c_{i}(t)x_{i}). \label{E:20}
    \end{align}
    So ~\eqref{E:20} shows that
    \begin{align*}
        \operatorname{var}(g(X)) & \ge \operatorname{var}(x_{N}) \\
        & \ge 3.
    \end{align*}
    However, this contradicts the fact that $ (g(X), h(\widetilde{X})) $ is a non-trivial, $ c(t) $-pair by Corollary ~\ref{Cor:pairSubSpace} (see page \pageref{Cor:pairSubSpace}).  Therefore the large pedestal ideal of $ k[X] $ is equal to zero.
\end{proof}
\section{Quasi-Principle Actions.}
Results from the previous section show that a representation whose large pedestal ideal is equal to zero is utterly uninteresting from the perspective of invariant theory and Hilbert's Fourteenth problem.  However, what about quasi-principle representations?  This section deals a lot with some of the tools that may be used for studying quasi-principle $ \mathbb{G}_{a} $-varieties.  If $ (g,1) $ is a principle pair for a $ \mathbb{G}_{a} $-variety $ \operatorname{Spec}(A) $ with action $ \beta $ over a field $ L $ of characteristic zero, then other authors have defined a map from $ A $ to $ A^{\mathbb{G}_{a}} $ which sends an element $ f $ to $ \beta^{\sharp}(f) \mid_{t = -g} $.  We prove that this map works in all characteristics.

If $ \operatorname{Spec}(A) $ is a quasi-principle variety with a quasi-principle $ b(t) $-pair $ (g,h) $, then we may assume that the action of $ \mathbb{G}_{a} $ on $ \operatorname{Spec}(A) $ is free.  The reason is that $ \mathbf{ker}(b(t)) $ acts trivially on $ \operatorname{Spec}(A) $ and $ (A^{\mathbb{G}_{a}/\mathbf{ker}(b(t))})^{\mathbf{ker}(b(t))} \cong A^{\mathbb{G}_{a}} $.
\begin{prop} \label{P:expandedSliceSep}
    If $ \operatorname{Spec}(A) $ is a $ \mathbb{G}_{a} $-variety with action $ \beta $, such that $ A $ is generated by
    $ z_{1},\dots,z_{n} $ as a $ k $-algebra, $ \beta^{\sharp}(z_{i}) $ is equal to $ v_{i}(Z,t) $, and $ (g,h) $ is a principle pair (see Definition ~\ref{D:aPair}), then the rational functions:
    \begin{equation} \label{E:82}
        f_{i}(Z) = v_{i}(z_{1},\dots,z_{n},-g/h) \quad 1 \le i \le n
    \end{equation}
    are invariant.
\end{prop}
\begin{proof}
    Let us denote $ \beta^{\sharp}(z_{i}) $ by $ v_{i}(Z,t) $.  If $ x \in \operatorname{Spec}(A) $ and $ s,w \in \mathbb{G}_{a} $, then since $ \operatorname{Spec}(A) $ is a $ \mathbb{G}_{a} $-variety:
    \begin{align}
        v_{i}(x,s+w) &= z_{i}((s+w) \ast x), \notag \\
        &= v_{i}(z_{1}(w \ast x),\dots,z_{n}(w \ast x),s). \label{E:87}
    \end{align}
    Equating the left and right sides of ~\eqref{E:87}:
    \begin{equation} \label{E:88}
        v_{i}(x,s+w) = v_{i}(z_{1}(w \ast x),\dots,z_{n}(w \ast x),s).
    \end{equation}
    If we substitute $ (-g/h)(w \ast x) $ for $ s $, then:
    \begin{align*}
        f_{i}(w \ast x) &= v_{i}(z_{1}(w \ast x),\dots,z_{n}(w \ast x),(-g/h)(w \ast x)) \\
        &= v_{i}(z_{1}(w \ast x),\dots,z_{n}(w \ast x),(-g/h)(x)-w) \\
        &= v_{i}(x,(-g/h)(x)-w+w) \\
        &= v_{i}(x,(-g/h)(x)) \\
        &= f_{i}(x),
        \end{align*}
    where the jump from the second line to the third uses ~\eqref{E:88}.  Because $ f_{i}(Z) $ is constant on the orbits of $ \mathbb{G}_{a} $, it is in $ A_{h}^{\mathbb{G}_{a}} $.
\end{proof}
\begin{prop} \label{P:itsTheRing}
    Let $ \operatorname{Spec}(A) $ be a $ \mathbb{G}_{a} $-variety with action $ \beta $ and let \linebreak $ z_{1},\dots,z_{n} $ generate $ A $ as a $ k $-algebra.  If $ (g, h) $ is a principle pair (see Definition ~\ref{D:aPair} on page \pageref{D:aPair}), $ v_{i}(Z,t) $ is equal to $ \beta^{\sharp}(z_{i}) $, and $ f_{i}(Z) $ equals $ v_{i}(Z,-g/h) $, then
    \begin{equation*}
        k[f_{1}(Z),\dots,f_{n}(Z)]_{h} = (A_{h})^{\mathbb{G}_{a}}.
    \end{equation*}
    Moreover, if $ r(Z) $ is an element of $ A^{\mathbb{G}_{a}}_{h} $, then it is equal to $ r(f_{1}(Z),\dots,f_{n}(Z)) $.
\end{prop}
\begin{proof}
    If $ (g,h) $ is a principle pair, and $ r(Z) \in (A_{h})^{\mathbb{G}_{a}} $, then:
    \begin{align}
        r(Z) &= \beta^{\sharp}(r(Z)), \notag \\
        &= r(\beta^{\sharp}(z_{1}),\dots,\beta^{\sharp}(z_{n})), \notag \\
        &=r\left( v_{1}(Z,t),\dots,v_{n}(Z,t) \right). \label{E:36}
    \end{align}
    Since $ r(Z) $ does not depend on the value of $ t $ in ~\eqref{E:36}, we may substitute $ (-g/h) $ for $ t $ in ~\eqref{E:36}.  So,
    \begin{align*}
        r(Z) &= r(v_{1}(Z,-g/h),\dots,v_{n}(Z,-g/h)) \\
        &=r(f_{1}(Z),\dots,f_{n}(Z)).
    \end{align*}
    Therefore $ r(Z) \in k[f_{1}(Z),\dots,f_{n}(Z)]_{h} $, i.e.,
    \begin{equation*}
        k[f_{1}(Z),\dots,f_{n}(Z)]_{h} = A^{\mathbb{G}_{a}}_{h}.
    \end{equation*}
\end{proof}
\section{When The Large Pedestal Ideal is Non-Zero, but the Pedestal Ideal is Trivial.}
\begin{thm} \label{T:trivialPedForm}
    Let $ \beta: \mathbb{G}_{a} \to \operatorname{GL}(\mathbf{V}) $ be an $ n $-dimensional, linear representation.  The following statements are equivalent:
    \begin{itemize}
        \item[a)] the large pedestal ideal of $ S_{k}(\mathbf{V}^{\ast}) $ is non-zero, but $ \mathfrak{P}(S_{k}(\mathbf{V}^{\ast})) $ is equal to zero,
        \item[b)] the following statements hold:
        \begin{itemize}
            \item[i)] There is an upper triangular basis $ \{x_{1},\dots,x_{n}\} $ and an additive polynomial $ b(t) $ such that if $ \beta^{\sharp}(x_{i}) $ is equal to $ \sum_{j=1}^{i} q_{i,j}(t) x_{j} $, then \linebreak $ q_{i,j}(t) \in k[b(t)] $ for $ i<n $.
            \item[ii)] If $ i=n $, then $ q_{n,j}(t) \in k[b(t)] $ if $ x_{j} \notin (\mathbf{V}^{\ast})^{\mathbb{G}_{a}} $.  Otherwise, there is a polynomial $ s_{n,j}(z) $ with no constant term and an additive polynomial $ d_{j}(t) $ such that $ q_{n,j}(t) $ is equal to $ s_{n,j}(b(t))+d_{j}(t) $.  Moreover, \linebreak $ d_{j}(t) \in \mathfrak{O} \setminus \mathfrak{O} \langle b(t) \rangle $ whenever $ d_{j}(t) $ is non-zero.
            \item[iii)] If $ S $ is the set of $ j $ such that $ x_{j} \in (\mathbf{V}^{\ast})^{\mathbb{G}_{a}} $ and $ d_{j}(t) $ is non-zero, then the $ k $-span of $ \{d_{j}(t)\}_{j \in S} $ is greater than one,
            \item[iv)] There is a $ b(t) $-pair $ (g(X),h(X)) $.
        \end{itemize}
    \end{itemize}
    If the large pedestal ideal is non-zero, but the pedestal ideal is equal to zero, then the characteristic of $ k $ is greater than zero.
\end{thm}
\begin{proof}
    Assume that a) holds.  Let $ \{x_{1},\dots,x_{n}\} $ be an upper triangular basis of $ \mathbf{V}^{\ast} $.  Since $ \mathfrak{P}_{g}(k[X]) \ne 0 $, pairs exist.  The Ore ring is a non-commutative Euclidean ideal domain, so the fundamental ideal $ \mathfrak{f} $ is principle.  Because the large pedestal ideal of $ k[X] $ is non-zero, there is some additive polynomial $ b(t) $ such that $ \mathfrak{f}_{k[X]} = \mathfrak{O} \langle b(t) \rangle $.  Recall the definition of the fundamental ideal in Definition ~\ref{D:freudenbergIdeal}.

    Let us denote restriction of the action of $ \mathbb{G}_{a} $ to $ \mathbf{ker}(b(t)) $ on a representation $ \mathbf{W} $ by $ \operatorname{res}^{\mathbb{G}_{a}}_{\mathbf{ker}(b(t))}(\mathbf{W}) $.

    We claim that there is a representation $ \mathbf{W} $ such that $ \mathfrak{P}_{g}(S_{k}(\mathbf{W}^{\ast})) $ is equal to zero and $ \operatorname{res}^{\mathbb{G}_{a}}_{\mathbf{ker}(b(t))}(\mathbf{W}^{\ast}) = \operatorname{res}^{\mathbb{G}_{a}}_{\mathbf{ker}(b(t))}(\mathbf{V}^{\ast}) $.  Suppose that this is not the case.  If this is so, then for any representation $ \gamma: \mathbb{G}_{a} \to \operatorname{GL}(\mathbf{W}) $ such that 
    \begin{equation*}
        \operatorname{res}^{\mathbb{G}_{a}}_{\mathbf{ker}(b(t))}(\mathbf{W}^{\ast}) = \operatorname{res}^{\mathbb{G}_{a}}_{\mathbf{ker}(b(t))}(\mathbf{V}^{\ast})
    \end{equation*}
    there is a $ d(t) $-pair $ (G(X),H(X)) $ for $ d(t) \notin \mathfrak{O}\langle b(t) \rangle $.

    Since $ (G(X),H(X)) $ is a $ d(t) $-pair under the co-action $ \gamma^{\sharp} $, there are polynomials $ q_{j}(z) \in \langle z \rangle k[z] $ and $ H_{j}(X) \in k[X] $ such that
    \begin{equation*}
        \beta^{\sharp}(G(X)) = G(X)+d(t)H(X)+\sum_{j=2}^{\ell} q_{j}(b(t)) H_{j}(X).
    \end{equation*}
    Since the fundamental ideal $ \mathfrak{f}_{k[X]} $ is generated by $ b(t) $, there is a non-trivial $ b(t) $-pair $ (g(X),h(X)) $ by Corollary ~\ref{Cor:fundamental}.  If $ \psi $ is the action on $ \operatorname{Spec}(k[s]) \times \mathbf{V} $ such that
    \begin{align*}
        \psi^{\sharp} & \mid_{k[X]} = \beta^{\sharp} \\
        \psi^{\sharp}(s) &= s-t,
    \end{align*}
    then the ideal $ \langle b(s)+g(X)/h(X) \rangle $ is $ \mathbb{G}_{a} $-stable.  Let $ G_{1}(X) $ be the rational function below:
    \begin{equation*}
        G_{1}(X) = G(X)+\sum_{j=2}^{\ell} q_{j}(-g(X)/h(X)) H_{j}(X).
    \end{equation*}
    If $ \xi \in \mathcal{V}(\langle b(s)+g(X)/h(X) \rangle) $, and we denote $ \beta^{\sharp}(G(X)) $ by $ u(X,t) $, then $ G_{1}(X) $ is also equal to $ u(X,\xi)-d(\xi)H(X) $.  Since $ \Psi $ is an action, if $ y \in D(h(X)) $, and $ w_{0},w_{1} \in \mathbb{G}_{a} $, then
    \begin{align}
        u(w_{1} \ast y,w_{0}) &= G(w_{0} \ast (w_{1} \ast y)), \notag \\
        &= G((w_{1}+w_{0}) \ast y), \notag \\
        &= u(y,w_{1}+w_{0}). \label{E:99}
    \end{align}
    Note that $ (\xi,-1) $ is a principle pair.  So, if $ x \in D(h(X)) $ and $ w \in \mathbb{G}_{a} $, then upon substituting $ w $ for $ w_{1} $, $ x $ for $ y $ and $ \xi(w \ast x) $ for $ w_{0} $ in ~\eqref{E:99}, the following calculations show that $ (G_{1}(X),H(X)) $ is a $ d(t) $-pair:
    \begin{align*}
        G_{1}(w \ast x) &= u(w \ast x,\xi(w \ast x))-d(\xi(w \ast x))H(w \ast x) \\
        &= u(x, \xi(w \ast x)+w)-d(\xi(w \ast x)) H(w \ast x) \\
        &= u(x,\xi(x)-w+w)-d(\xi(w \ast x))H(w \ast x) \\
        &= u(x, \xi(x))-d(\xi(w \ast x)) H(w \ast x) \\
        &= u(x, \xi(x))-d(\xi(w \ast x)) H(x) \\
        &= u(x, \xi(x))-(d(\xi(x)) -d(w))H(x) \\
        &= u(x, \xi(x))-d(\xi(x))H(x)+d(w)H(x) \\
        &= G_{1}(x)+d(w)H(x),
    \end{align*}
    where we go from the second to third line using ~\eqref{E:99}.
    There are polynomials $ (r(X),s(X)) $ such that $ r(X)/s(X) $ is equal to $ G_{1}(X)/H(X) $.  So, by Lemma ~\ref{L:pairsLemma} (see page \pageref{L:pairsLemma})
    \begin{align*}
        \mathfrak{f} &\supseteq \mathfrak{O}(d(t),b(t)) \\
        & \supsetneq \mathfrak{O} \langle b(t) \rangle.
    \end{align*}
    This contradicts the fact that $ \mathfrak{f} = \mathfrak{O} \langle b(t) \rangle $.  Therefore, there is a representation $ \gamma: \mathbb{G}_{a} \to \operatorname{GL}(\mathbf{W}) $ such that $ \mathfrak{P}_{g}(S_{k}(\mathbf{W}^{\ast})) $ is equal to zero and
    \begin{equation*}
        \operatorname{res}^{\mathbb{G}_{a}}_{\mathbf{ker}(b(t))}(\mathbf{W}^{\ast}) = \operatorname{res}^{\mathbb{G}_{a}}_{\mathbf{ker}(b(t))}(\mathbf{V}^{\ast}).
    \end{equation*}

    If we apply Theorem ~\ref{T:trivialLargePedForm} (see page ~\pageref{T:trivialLargePedForm}) to the representation $ \mathbf{W} $, then there is an upper triangular basis $ \{x_{1},\dots,x_{n}\} $ such that the socle of $ \mathbf{W}^{\ast} $ is generated by $ \{x_{1},\dots,x_{n-1}\} $ and $ \gamma^{\sharp}(x_{n}) $ is equal to $ x_{n}+\sum_{j=1}^{n-1} d_{j}(t)x_{j} $ where $ d_{j}(t) \notin \mathfrak{O} \langle b(t) \rangle $ if it is non-zero.  As a result,
    \begin{enumerate}
        \item if $ \beta^{\sharp}(x_{i}) $ is equal to $ \sum_{j=1}^{i} q_{i,j}(t)x_{j} $ for $ q_{i,j}(t) \in k[t] $, then $ q_{i,j}(t) \in k[b(t)] $ for $ i<n $,
        \item for $ 1 \le j <n $, there is a polynomial $ s_{n,j}(z) $ with no constant term and an additive polynomial such that $ q_{n,j}(t) $ is equal to $ s_{n,j}(b(t)) +d_{j}(t) $.
        \item if $ d_{j}(t) $ is non-zero, then $ d_{j}(t) \notin \mathfrak{O} \langle b(t) \rangle $.
        \item the variance $ \operatorname{var}^{\mathbf{ker}(b(t))}(x_{n}) \ge 3 $.
    \end{enumerate}

    We may assume that $ \{x_{1},\dots,x_{c}\} $ is a basis of $ (\mathbf{V}^{\ast})^{\mathbb{G}_{a}} $.  Once again let $ \psi $ be the action on $ \operatorname{Spec}(k[s]) \times \mathbf{V} $ such that
    \begin{align*}
        \psi^{\sharp} \mid_{k[X]} &= \beta^{\sharp} \\
        \psi^{\sharp}(s) &= s-t.
    \end{align*}
    By our previous work $ \psi $ acts on $ \mathcal{V}(\langle b(s)+g(X)/h(X) \rangle $.  Since $ \psi $ is a co-action
    \begin{align*}
        v_{n}(X,t_{1}+t_{2}) &= (\operatorname{id}_{k[X]} \otimes \mu_{\mathbb{G}_{a}}^{\sharp}) \circ \psi^{\sharp}(x_{n}) \\
        &= (\psi^{\sharp} \otimes \operatorname{id}_{t_{1}}) \circ \psi^{\sharp}(x_{n}) \\
        &= v_{n}(v_{1}(X,t_{2}),\dots,v_{n}(X,t_{2}),t_{1}).
    \end{align*}
    We may simplify this as
    \begin{equation} \label{E:100}
        v_{n}(X,t_{1}+t_{2}) = v_{n}(v_{1}(X,t_{2}),\dots,v_{n}(X,t_{2}),t_{1}).
    \end{equation}
    Let $ g_{1}(X) $ equal $ x_{n}+\sum_{j=1}^{n-1} s_{n,j}(-g(X)/h(X))x_{j} $.  If $ \xi $ is an element of $ \mathcal{V}(\langle b(s)+g(X)/h(X) \rangle) $, then $ g_{1}(X) $ is also equal to $ v_{n}(X,\xi)-\sum_{j=1}^{n} d_{j}(\xi)x_{j} $.  Upon substituting $ \xi $ in for $ t_{1} $ and $ t $ for $ t_{2} $ in ~\eqref{E:100} we obtain
    \begin{align*}
        \psi^{\sharp}(v_{n}(X,\xi)) &= v_{n}(v_{1}(X,t),\dots,v_{n}(X,t),\xi-t) \\
        &= v_{n}(X,\xi-t+t) \\
        &= v_{n}(X,\xi).
    \end{align*}
    Therefore,
    \begin{align*}
        \beta^{\sharp}(g_{1}(X)) &=\psi^{\sharp}(g_{1}(X)) \\
        &= \psi^{\sharp}\left(v_{n}(X,\xi)-\sum_{j=1}^{n-1} d_{j}(\xi)x_{j}\right) \\
        &= v_{n}(X,\xi)- \psi^{\sharp}\left(\sum_{j=1}^{n-1}d_{j}(\xi)x_{j}\right) \\
        &= v_{n}(X,\xi)-\sum_{j=1}^{n-1} d_{j}(\xi-t)\sum_{i=1}^{j} q_{j,i}(t)x_{i} \\
        &= v_{n}(X,\xi)-\sum_{j=1}^{n-1}d_{j}(\xi)\left(\sum_{i=1}^{j} q_{j,i}(t)x_{i}\right) +\sum_{j=1}^{n-1} d_{j}(t)\left(\sum_{i=1}^{j} q_{j,i}(t)x_{i}\right) \\
        &= g_{1}(X) -\sum_{j=1}^{n-1} d_{j}(\xi)\left(\sum_{i=1}^{j-1}q_{j,i}(t)x_{i}\right)+\sum_{j=1}^{n-1} d_{j}(t)\left(\sum_{i=1}^{j} q_{j,i}(t)x_{i}\right),
    \end{align*}
    which would mean that $ \beta^{\sharp}(g_{1}(X)) \notin k[X]_{h(X)}[t] $ if $ d_{j}(t) $ is not equal to zero for all $ c<j <n $.  As a result, $ q_{n,j}(t) \in k[b(t)] $ for $ j>c $.
      
    Let $ S $ be the set of $ 1 \le j \le c $ such that $ d_{j}(t) \ne 0 $.  Because $ \operatorname{var}^{\mathbf{ker}(b(t))}(x_{n}) $ the $ k $-span of $ \{d_{j}(t)\}_{j \in S} $ is greater than one.
    If conditions i), ii), iii) and iv) of b) hold, then $ \mathfrak{P}_{g}(k[X]) \ne 0 $ because there is a $ b(t) $-pair.  Let us assume that the dimension of the $ k $-span of $ \{d_{j}(t)\}_{j\in S} $ is $ m $.  After a change of basis we may assume that $ q_{i,j}(t) \in k[b(t)] $ unless $ i=n $ and $ 1 \le j \le m $.  Moreover we may assume that whenever $ 1 \le j \le m $, there is a polynomial $ s_{n,j}(z) $ with no constant term and an additive polynomial $ d_{j}(t) \notin \mathfrak{O}\langle b(t) \rangle $ such that $ q_{n,j}(t) $ is equal to $ s_{n,j}(b(t))+d_{j}(t) $.

    If $ x $ is a point of $ \mathbf{V}^{as} \cap D(h(X)) $, then $ (\mathbb{G}_{a})_{x} \subseteq \mathbf{ker}(b(t)) $ since any $ t_{0} \in (\mathbb{G}_{a})_{x} $ has the property that
    \begin{align*}
        g(x) &= g(t_{0} \ast x) \\
        &= g(x)+b(t_{0})h(x).
    \end{align*}
    The only other condition for a point $ t_{0} \in \mathbf{ker}(b(t)) $ to stabilize a point $ x $ equal to $ (a_{0},\dots,a_{m-1},\dots,a_{n-1}) $ is for $ \sum_{i=1}^{m} a_{i-1}d_{i}(t_{0}) $ to equal zero.  There is an open sub-variety of points $ y $ equal to $ (c_{0},\dots,c_{m-1},\dots,c_{n-1}) $ such that
    \begin{equation*}
        \sum_{i=1}^{m} a_{i-1}d_{i}(t) \ne \sum_{i=1}^{m} c_{i-1}d_{i}(t).
    \end{equation*}
    Since
    \begin{align*}
        (\mathbb{G}_{a})_{x} &= \mathbf{ker}(b(t)) \cap \mathcal{V}(\langle \sum_{i=1}^{m} a_{i-1} d_{i}(t) \rangle) \\
        (\mathbb{G}_{a})_{y} &= \mathbf{ker}(b(t)) \cap \mathcal{V}(\langle \sum_{i=1}^{m} c_{i-1} d_{i}(t) \rangle),
    \end{align*}
    there cannot be an open sub-variety $ \mathbf{V}^{as} $ of $ \mathbf{V} $ such that the stabilizer of any point is identical.  Therefore a quasi-principle pair cannot exist, i.e, $ \mathfrak{P}(k[X]) $ is equal to zero.  So a) and b) are equivalent conditions.
\end{proof}
Assume that $ \beta: \mathbb{G}_{a} \to \operatorname{GL}(\mathbf{V}) $ is an $ n $-dimensional, linear representation such that $ \mathfrak{P}_{g}(S_{k}(\mathbf{V}^{\ast})) $ is non-zero, but $ \mathfrak{P}(S_{k}(\mathbf{V}^{\ast})) $ is equal to zero.  By Theorem ~\ref{T:trivialPedForm}, there is a basis $ \{x_{1},\dots,x_{c}\} $ of $ (\mathbf{V}^{\ast})^{\mathbb{G}_{a}} $ and a basis $ \{x_{1},\dots,x_{n}\} $ of $ \mathbf{V}^{\ast} $ such that 
\begin{itemize}
    \item[i)] there is an additive polynomial $ b(t) $ such that if $ \beta^{\sharp}(x_{i})= \sum_{j=1}^{i} q_{i,j}(t)x_{j} $, then $ q_{i,j}(t) \in k[b(t)] $ whenever $ i<n $ or $ i=n $ and $ j>c $,
    \item[ii)] if $ i=n $ and $ 1 \le j \le c $, then there is a polynomial $ s_{n,j}(z) $ with no non-constant term and an additive polynomial $ d_{j}(t) $ such that $ q_{n,j}(t) $ is equal to $ s_{n,j}(b(t)) +d_{j}(t) $. Moreover, $ d_{j}(t) \notin \mathfrak{O} \langle b(t) \rangle $ if it is non-zero.
    \item[iii)] the dimension of the span of $ \{d_{j}(t)\}_{j=1}^{s} $ as a $ k $-vector space is at least two,
    \item[iv)] there is a $ b(t) $-pair $ (g(X),h(X)) $.
\end{itemize}
Let us denote $ \beta^{\sharp}(x_{i}) $ by $ v_{i}(X,b(t)) $ if $ i<n $ and $ \beta^{\sharp}(x_{n}) $ by $ v_{n}(X,b(t)) +u(X,t) $ where $ u(X,t) $ is equal to $ \sum_{j=1}^{c} d_{j}(t)x_{j} $ and $ v_{n}(X,b(t)) $ is equal to $ \beta^{\sharp}(x_{n})-u(X,t) $.  Let $ f_{i}(X) $ equal $ v_{i}(X,-g(X)/h(X)) $ and let $ f_{n}(X) $ equal $ v_{n}(X,-g(X)/h(X)) $.
\begin{lem} \label{L:startToNoPedInv}
    Let $ L $ be the splitting field of $ b(s)+g(X)/h(X) $ over $ k(X) $ and let $ L^{sep} $ be the separable closure of $ k(X) $ in $ L $.  If $ A $ is the integral closure of $ k[X] $ in $ L^{sep} $, then the variety $ \operatorname{Spec}(A) $ is a generically principle, $ \mathbb{G}_{a} $-variety and the natural morphism $ \tau: \operatorname{Spec}(A) \to \mathbf{V} $ is a generically finite, $ \mathbb{G}_{a} $-equivariant morphism.
\end{lem}
\begin{proof}
    Let $ \psi $ be the action of $ \mathbb{G}_{a} $ on $ \mathbf{V} \times \operatorname{Spec}(k[s]) $ such that
    \begin{align*}
        \psi^{\sharp} \mid_{k[X]} &= \beta^{\sharp} \\
        \psi^{\sharp}(s) &= s-t.
    \end{align*}
    Because
    \begin{align*}
        \psi^{\sharp}(b(s)+g(X)/h(X)) &= \psi^{\sharp}(b(s))+\psi^{\sharp}(g(X)/h(X)) \\
        &= b(s-t)+\psi^{\sharp}(g(X)/h(X)) \\
        &= b(s-t)+\left(g(X)/h(X)+b(t)\right) \\
        &= b(s)-b(t)+g(X)/h(X)+b(t) \\
        &= b(s)+g(X)/h(X),
    \end{align*}
    the ideal $ \langle b(s)+g(X)/h(X) \rangle k[X,s]_{h(X)} $ is $ \mathbb{G}_{a} $-stable.  Therefore, the action $ \psi $ is well defined on $ \operatorname{Spec}(A) $.  Since $ \psi^{\sharp} \mid_{k[X]} =\beta^{\sharp} $, the morphism $ \tau: \operatorname{Spec}(A) \to \mathbf{V} $ is a generically finite, $ \mathbb{G}_{a} $-equivariant morphism.
    
    Assume that the separable degree of $ b(s)+g(X)/h(X) $ as a polynomial in $ k(X)[s] $ is $ p^{v} $.  There are $ 2p^{v} $ elements $ \{\phi_{i},\eta_{i}\}_{i=1}^{p^{v}} $ of $ A $ and natural numbers $ e_{1},\dots,e_{p^{v}} \in \mathbb{N} $ such that
    \begin{equation*}
        b(s)+g(X)/h(X) = \prod_{i=1}^{p^{v}}(\eta_{i}+\phi_{i}s)^{e_{i}}.
    \end{equation*}
    The co-action $ \psi^{\sharp} $ leaves $ \phi_{i} $ invariant for $ 1 \le i \le p^{v} $ and sends $ \eta_{i} $ to $ \eta_{i}+t\phi_{i} $.  Because $ (\eta_{i},\phi_{i}) $ is a principle pair for any $ 1 \le i \le p^{v} $, the variety $ \operatorname{Spec}(A) $ is generically principle.
\end{proof}
\begin{thm}
    Let us adopt the conventions of Lemma ~\ref{L:startToNoPedInv}.  Let $ H $ be the Galois group of $ L^{sep} $ over $ k(X) $ and let $ r_{i} $ equal $ f_{n}(X)+ u_{n}(X,-\eta_{i}/\phi_{i}) $ for $ 1\le i\le p^{v} $.  The ring $ k[X]_{h(X)}^{\mathbb{G}_{a}} $ is equal to
    \begin{align*}
        k[f_{1}(X),\dots,f_{n-1}(X),r_{1},\dots,r_{p^{c}}]^{H}.
    \end{align*}
\end{thm}
\begin{proof}
    Because $ A $ is equal to $ k[x_{1},\dots,x_{n},\phi_{i},\eta_{i}]_{i=1}^{p^{v}} $, the ring $ A_{\phi_{i}}^{\mathbb{G}_{a}} $ is equal to
    \begin{equation*}
        A_{\phi_{i}}^{\mathbb{G}_{a}} = \left(k[f_{1}(X),\dots,f_{n-1}(X),r_{i}, \eta_{j} \phi_{i}-\eta_{i}\phi_{j}]_{j\ne i}\right)_{\phi_{i}},
    \end{equation*}
    by Proposition ~\ref{P:itsTheRing}.

    Observe that by our construction the ring $ A^{H} \subseteq k[X] $.  Let $ u(X) $ be an element of $ k[X]_{h(X)}^{\mathbb{G}_{a}} $. Because $ \beta^{\sharp} $ is an action
    \begin{align}
        u(X) &= \beta^{\sharp}(u(X)),\notag \\
        &= u\left(v_{1}(X,b(t)),\dots,v_{n-1}(X,b(t)),v_{n}(X,b(t))+u_{n}(X,t)\right) \label{E:95}
    \end{align}
    The left hand side of ~\eqref{E:95} does not depend on $ t $, so we may substitute $ -\eta_{i}/\phi_{i} $ for $ t $. If we do so, then
    \begin{align*}
        u(X)&= u(f_{1}(X),\dots,f_{n-1}(X),r_{i}) \\
        &\in k[f_{1}(X),\dots,f_{n-1}(X),r_{1},\dots,r_{p^{v}}]^{H}.
    \end{align*}
    Therefore,
    \begin{equation*}
        k[X]^{\mathbb{G}_{a}}_{h(X)} \subseteq k[f_{1}(X),\dots,f_{n-1}(X),r_{1},\dots,r_{p^{v}}]^{H}.
    \end{equation*}
    Observe that $ r_{1},\dots,r_{p^{c}} $ are all invariant so,
    \begin{align*}
        k[f_{1}(X),\dots,f_{n-1}(X),r_{1},\dots,r_{p^{v}}]^{H} &\subseteq \left(A_{\phi_{1}\cdots \phi_{p^{v}}}^{\mathbb{G}_{a}}\right)^{H} \\
        &= (A_{h(X)}^{\mathbb{G}_{a}})^{H} \\
        & \subseteq A_{h(X)}^{\mathbb{G}_{a}} \cap A^{H} \\
        &= A_{h(X)}^{\mathbb{G}_{a}} \cap k[X] \\
        &=k[X]_{h(X)}^{\mathbb{G}_{a}}.
    \end{align*}.
    
    Therefore,
    \begin{equation*}
        k[X]_{h(X)}^{\mathbb{G}_{a}} = k[f_{1}(X),\dots,f_{n-1}(X),r_{1},\dots,r_{p^{v}}]^{H}.
    \end{equation*}
\end{proof} 
\bibliographystyle{amsplain}
\bibliography{TheoryOfGAActions}

\providecommand{\bysame}{\leavevmode\hbox to3em{\hrulefill}\thinspace}
\providecommand{\MR}{\relax\ifhmode\unskip\space\fi MR }
\providecommand{\MRhref}[2]{%
  \href{http://www.ams.org/mathscinet-getitem?mr=#1}{#2}
}
\providecommand{\href}[2]{#2}
\begin{thebibliography}{1}

\bibitem{FreudenbergDerivBook}
Gene Freudenburg et~al., \emph{Algebraic theory of locally nilpotent
  derivations}, vol. 136, Springer, 2006.

\bibitem{Goss}
David Goss, \emph{Basic structures of function field arithmetic}, Springer
  Science \& Business Media, 2012.

\bibitem{HartshorneAG}
Robin Hartshorne, \emph{Algebraic geometry}, Springer-Verlag, New York, 1977,
  Graduate Texts in Mathematics, No. 52. \MR{MR0463157 (57 \#3116)}

\bibitem{Kuroda}
Shigeru Kuroda, \emph{A generalization of nakai's theorem on locally finite
  iterative higher derivations}, Osaka Journal of Mathematics \textbf{54}
  (2017), no.~2, 335--341.

\bibitem{Ore}
Oystein Ore, \emph{On a special class of polynomials}, Transactions of the
  American Mathematical Society \textbf{35} (1933), no.~3, 559--584.

\bibitem{RosenlichtCS}
Maxwell Rosenlicht, \emph{Some basic theorems on algebraic groups}, American
  Journal of Mathematics \textbf{78} (1956), no.~2, 401--443.

\bibitem{Rotman}
Joseph~J Rotman and Joseph~Jonah Rotman, \emph{An introduction to homological
  algebra}, vol.~2, Springer, 2009.

\bibitem{Essen}
Arno Van~den Essen, \emph{Polynomial automorphisms: and the jacobian
  conjecture}, vol. 190, Springer Science \& Business Media, 2000.

\end{thebibliography}
\end{document}